\documentclass[english]{article}
\usepackage[T1]{fontenc}
\usepackage[latin9]{inputenc}
\usepackage{verbatim}
\usepackage{amsmath}
\usepackage{amssymb}
\PassOptionsToPackage{normalem}{ulem}
\usepackage{ulem}

\makeatletter

\makeatletter
\usepackage[all,arc,curve,color,frame]{xy}

\newxyColor{pink}{1.0 0.4 0.5}{rgb}{} 

\usepackage{amsthm}
\usepackage{amsfonts}
\usepackage{bbm}
\usepackage{enumerate}

\allowdisplaybreaks

\newtheorem{theorem}{Theorem}[section]
\newtheorem{lemma}[theorem]{Lemma}
\newtheorem{notation}[theorem]{Notation}

\newtheorem{hypothesis}{Hypothesis}
\newtheorem{proposition}[theorem]{Proposition}
\newtheorem{definition}[theorem]{Definition}
\newtheorem{remark}[theorem]{Remark}
\newtheorem{corollary}[theorem]{Corollary}
\newtheorem{example}[theorem]{Example}
\numberwithin{equation}{section}

\renewcommand{\1}{\mathbbm{1}}

\newcommand{\N}{\mathbb{N}}
\newcommand{\Z}{\mathbb{Z}}

\newcommand{\E}{\mathbb{E}}

\newcommand{\p}{\mathbb{P}}
\newcommand{\R}{\mathbb{R}}

\newcommand{\convN}{\underset{N\rightarrow +\infty}{\longrightarrow}}

\renewcommand{\epsilon}{ \varepsilon}

\newcommand{\convloi}{\underset{N\rightarrow +\infty}{\overset{\text{law}}{\longrightarrow}}}

\newcommand{\Bsym}{\mathcal{B}^{\text{sym}}_0}

\newcommand{\bZ}{\overline{Z}}
\newcommand{\tZ}{\widetilde{Z}}
\newcommand{\eqlaw}{\overset{\text{law}}{=}}

\newcommand{\tK}{\widetilde{K}}
\newcommand{\tT}{\widetilde{T}}

\newcommand{\ktt}{\widetilde{K}_t}
\newcommand{\kt}{K_t}

\makeatother

\usepackage{a4wide}

\makeatother

\usepackage{babel}
\begin{document}

\title{Central limit theorem through expansion of the propagation of chaos
for Bird and Nanbu systems}

\author{Sylvain Rubenthaler\thanks{Univ. Nice Sophia Antipolis, CNRS, LJAD, UMR 7351, 06100 Nice, France.
E-mail: rubentha@unice.fr.}}
\maketitle
\begin{abstract}
The Bird and Nanbu systems are particle systems used to approximate
the solution of the mollified Boltzmann equation. These systems have
the propagation of chaos property. Following \cite{graham-meleard-1994,graham-meleard-1997,graham-meleard-1999},
we use coupling techniques  to write a kind of expansion of the error
in the propagation of chaos in terms of the number of particles.
This expansion enables us to prove the a.s convergence and the central-limit
theorem for these systems. Notably, we obtain a central-limit theorem
for the empirical measure of the system. As it is the case in \cite{graham-meleard-1994,graham-meleard-1997,graham-meleard-1999},
these results apply to the trajectories of particles on an interval
$[0,T]$.

Les systèmes de Bird et Nanbu sont des systèmes de particules en interaction
approchant la solution de l'équation de Boltzmann mollifiée. Ces systèmes
vérifient la propagation du chaos. Dans l'esprint de \cite{graham-meleard-1994,graham-meleard-1997,graham-meleard-1999},
nous utilisons des techniques de couplage pour écrire un développement
asymptotique dans la propagation du chaos, en terme du nombre de particules.
Ce développement nous permet de démontre la convergence p.s. de ces
systèmes, ainsi qu'un théorème central-limite. Ce théorème central-limite
s'applique à la mesure empirique du système. Comme dans \cite{graham-meleard-1994,graham-meleard-1997,graham-meleard-1999},
ces résultats s'appliquent aux trajectoires des particules sur un
intervalle $[0;T]$,


\textbf{Mots-clés~: }système de particules en interaction, équation
de Boltzmann, diffusion non-linéaire avec sauts, graphes et arbres
aléatoires, coupalge, propagation du chaos, Monte-Carlo, $U$-statistiques,
limite gaussienne, champ gaussien. 

\end{abstract}

\section{Introduction}

In  (\cite{del-moral-patras-rubenthaler-2008}), we obtained an expansion
of the propagation of chaos for a Feynman-Kac particle system (which
means, in the context \cite{del-moral-patras-rubenthaler-2008}, that
the particles are interacting through a ``selection of the fittest''
process). This particle system approximates a particular Feynman-Kac
measure, in the sense that the empirical measure associated to the
system converges to the Feynman-Kac measure when the number of particles
$N$ goes to $\infty$. What is called propagation of chaos is the
following double property of the particle system:
\begin{itemize}
\item $q$ particles, amongst the total of $N$ particles, looked upon at
a fixed time, are asymptotically independent when $N\rightarrow+\infty$
($q$ is fixed) 
\item and their law is converging to the Feynman-Kac law. 
\end{itemize}
In \cite{del-moral-patras-rubenthaler-2008}, we wrote an expansion,
in powers of $N$, of the difference between the law of $q$ independent
particles, each of them of the Feynman-Kac law, and the law of $q$
particles coming from the particle system. This expansion can be called
a functional representation like in \cite{del-moral-patras-rubenthaler-2008};
in the present paper, we call it an expansion of the error in the
propagation of chaos. In the setting of \cite{del-moral-patras-rubenthaler-2008},
the time is discrete. In \cite{del-moral-patras-rubenthaler-2008},
we showed how to use this kind of expansion to derive a.s. convergence
results (p. 824). In \cite{del-moral-patras-rubenthaler-2009}, we
extend the result of \cite{del-moral-patras-rubenthaler-2008} to
the case where the time is continuous, still in the Feynman-Kac framework,
and we establish central-limit theorems for $U$-statistics of these
systems of particles. The proof of the central-limit theorems for
$U$-statistics relies only on the exploitation of the expansion mentioned
above.

In this paper, our aim is to establish a similar expansion for a family
of particles systems including Bird and Nanbu systems. We do not go
as far as obtaining an expansion in the terms of Theorem 1.6 and Corollary
1.8 of \cite{del-moral-patras-rubenthaler-2009}, but our expansion
is sufficient to prove central-limit theorems (Theorem \ref{Th:TCL-debut}
and Corollary \ref{cor:TCL}). Bird and Nanbu systems are used to
approximate the solution of the mollified Boltzmann equation. We refer
mainly to \cite{graham-meleard-1997} and take into account models
described in (2.5), (2.6) of \cite{graham-meleard-1997} (a similar
description can be found in \cite{graham-meleard-1999}, Section 3).
Another reference paper on the subject is \cite{graham-meleard-1994}.
Our paper is mainly interesting in the following: it provides a sequel
to the estimates on propagation of chaos of \cite{graham-meleard-1997},
\cite{graham-meleard-1999} and it allows to apply the techniques
of \cite{del-moral-patras-rubenthaler-2008}, \cite{del-moral-patras-rubenthaler-2009}
to Bird and Nanbu systems. In particular:
\begin{itemize}
\item In the present paper, we obtain a central-limit theorem for the empirical
measure of the system (Th. \ref{Th:TCL-debut}) under less assumptions
than in \cite{meleard-1998} Th. 4.2, 4.3. (we only make assumptions
that are sufficient to ensure a solution to the problem \ref{def:pb-02}
defined in Definition \ref{def:pb-02}). Note that the results of
\cite{meleard-1998} hold under the assumption that the operator $L$,
describing the ``free'' trajectories of the particles (see below),
has a certain form, and that its coefficients and their derivatives
up to a certain order are bounded (see in particular $(H_{0}'')$
p. 215 of \cite{meleard-1998}). These assumptions are stronger than
our and are more than what is required to have existence of a solution
to \ref{def:pb-02}. Note also that the result in \cite{meleard-1998}
is a functional CLT for the empirical process whereas our result is
a Gaussian fluctuation field result for the empirical measure, it
considers only a finite number of centered real test functions. A
result similar to \cite{meleard-1998} can be found in \cite{uchiyama-1983-b}
(with similar assumptions).
\item Our convergence results (Theorem \ref{Th:conv-ps-debut}, Theorem
\ref{Th:TCL-debut}, Corollary \ref{cor:TCL}) hold for particles
trajectories on any interval $[0,T]$. 
\end{itemize}
Here, the proofs are radically different from those in \cite{del-moral-patras-rubenthaler-2009}
and this is why we decided to write them in a different paper. In
\cite{del-moral-patras-rubenthaler-2009}, we deal with combinatorial
problems related to the particle system studied there whereas in the
present paper, we deal with coupling problems. 

In Section \ref{sec:Definition-and-main}, we will present  Bird and
Nanbu models, as they can be found in \cite{graham-meleard-1997}
and we will state our main results: Theorem \ref{Theo:boltzmann}
is a refinement of the propagation of chaos results for the above-cited
models, Theorem \ref{Th:conv-ps-debut} is an a.s. convergence result
for these systems and Theorem \ref{Th:TCL-debut} and Corollary \ref{cor:TCL}
are central-limit theorems for these systems. In Section  \ref{sec:Other-systems-of},
we will introduce various particle systems which will be useful in
the proofs and we will prove Th. \ref{Theo:boltzmann}. The proof
of Th. \ref{Theo:boltzmann} relies on estimates on population growth
found in \cite{athreya-ney-2004} and on coupling ideas. In Section
\ref{sec:Rate-of-convergence}, we will prove a convergence result
for a particular kind of centered functions (Proposition \ref{Prop:convq2}),
from which we will deduce Corollary \ref{Cor:wickproduct}. The kind
of result found in Corollary \ref{Cor:wickproduct} is called a Wick-type
formula in \cite{del-moral-patras-rubenthaler-2008} (see (3.6) p.
807 in \cite{del-moral-patras-rubenthaler-2008} and \cite{del-moral-patras-rubenthaler-2009},
p.15 and Proposition \ref{Prop:convq2}). Corollary \ref{Cor:wickproduct}
and Proposition \ref{Prop:convq2} are used in Section \ref{Sec:conv}
to prove Th. \ref{Th:conv-ps-debut} and Th. \ref{Th:TCL-debut} and
Cor. \ref{cor:TCL}. Similar results can be found in \cite{dawson-zheng-1991,shiga-tanaka-1985,uchiyama-1983-a,uchiyama-1988}.
We will compare them to our result after the statement of Th. \ref{Th:TCL-debut}. 

Note that CLT's of the same kind as our can be found in \cite{sznitman-1984,sznitman-1985}.
The equation approximated by particles systems in these papers are
quite different from our limit equation.

An important point is that here we want to discuss the mathematical
properties of a certain class of particle systems. We will not discuss
the physical models. Such a discussion can be found in \cite{graham-meleard-1999}.

\section{Definition and main results\label{sec:Definition-and-main}}

\subsection{A first particle model}

\label{Sec:definition}

 In the following, we deal with particles evolving in $\R^{d}$.
We set the mappings $e_{i}:h\in\R^{d}\mapsto e_{i}(h)=(0,\dots,0,h,0,\dots,0)\in\R^{d\times N}$
($h$ at the $i$-th rank) ($1\leq i\leq N$). We have a Markov generator
$L$ and a kernel $\widehat{\mu}(v,w,dh,dk)$ on $\R^{2d}$ which
is symmetrical (that is $\widehat{\mu}(v,w,dh,dk)=\widehat{\mu}(w,v,dk,dh)$).
We set $\mu(v,w,dh)$ to be the marginal $\widehat{\mu}(v,w,dh\times\R^{d})$
up to mass at zero. Our assumptions are the same as in \cite{graham-meleard-1997}:

\begin{hypothesis}\label{Hyp:01}
\begin{enumerate}
\item We suppose that the generator $L$ on $\R^{d}$ acts on a domain $\mathcal{D}(L)$
of $L^{\infty}(\R^{d})$. (See \cite{graham-meleard-1997} p. 119
for a discussion on $\mathcal{D}(L)$). \label{H1} 
\item We suppose $\sup_{x,a}\widehat{\mu}(x,a,\R^{d}\times\R^{d})\leq\Lambda<\infty$.
\label{H2} 
\end{enumerate}
\end{hypothesis}

In Nanbu and Bird systems, the kernel $\widehat{\mu}$ and the generator
$L$ have specific features coming from physical considerations. In
these systems, the coordinates in $\R^{d}$ represent the position
and speed of molecules. However, these considerations have no effect
on our proof. That is why we claim to have a proof for systems more
general than Bird and Nanbu systems.

%
%
%

The Nanbu and Bird systems are defined in (2.5) and (2.6) of \cite{graham-meleard-1997},
by the means of integrals over Poisson processes. Here, we give an
equivalent definition.

\begin{definition}\label{Def:system} The particle system described
in \cite{graham-meleard-1997} is denoted by 
\[
(\bZ_{t})_{t\geq0}=(\bZ_{t}^{i})_{t\geq0,1\leq i\leq N}\ .
\]
 It is a process of $N$ particles in $\R^{d}$ and can be summarized
by the following.
\begin{enumerate}
\item Particles $(\bZ_{0}^{i})_{1\leq i\leq N}$ in $\R^{d}$ are drawn
i.i.d. at time $0$ according to a law $\widetilde{P}_{0}$. 
 
\item Between jump times, the particles evolve independently of each other
according to $L$. 
\item \label{Def:point3}We have a collection $(N_{i,j})_{1\leq i<j\leq N}$
of independent Poisson processes of parameter $\Lambda/(N-1)$ (the
parameter $\Lambda$ coming from Hypothesis \ref{H1}). For $i>j$,
we set $N_{i,j}=N_{j,i}$. If $N_{i,j}$ has a jump at time $t$,
we say that there is an \underline{interaction} between particles
$i$ and $j$. If there is an interaction at time $t$, then the system
undergoes a \underline{jump} with probability $\frac{\widehat{\mu}(\bZ_{t-}^{i},\bZ_{t-}^{j},\R^{2d})}{\Lambda}$
:
\begin{equation}
\bZ_{t}=\bZ_{t-}+e_{i}(H)+e_{j}(K)\,,\,\mbox{with\,}\,(H,K)\sim\frac{{\widehat{\mu}}(\bZ_{t-}^{i},\bZ_{t-}^{j},.,.)}{\widehat{\mu}(\bZ_{t-}^{N,i},\bZ_{t-}^{N,j},\R^{2d})}\label{Eq:BirdJump}
\end{equation}
 (independently of all the other variables). \\
And with probability $1-\frac{\widehat{\mu}(\bZ_{t-}^{i},\bZ_{t-}^{j},\R^{2d})}{\Lambda}$
, there is \uline{no jump} at time $t$. We will use $(\bZ_{0:t}^{i})_{1\leq i\leq N}$
to denote the system of the trajectories of particles on $[0,t]$
($\forall t\geq0$), that is for all $i$: $\bZ_{0:t}^{i}=(\bZ_{s}^{i})_{0\leq s\leq t}$.
We will use this notation ``$0:t$'' again in the following for
the same purpose.
\end{enumerate}
\end{definition}

We denote the Skorohod space of processes in $\R^{d}$ by $\mathbb{D}(\R^{+},\R^{d})$
(or $\mathbb{D}([0;t],\R^{d})$, depending of the domain). As in \cite{graham-meleard-1994},
we define the total variation norm by the following: for all signed
measures $\nu$ on a measurable space $(S,\mathcal{S})$,
\[
\Vert\nu\Vert_{TV}=\sup\left\{ \int_{S}f(x)\nu(dx),\Vert f\Vert_{\infty}\leq1\right\} \,.
\]
 \begin{definition}\label{def:pb-02}Let $\xi$ be the canonical process
on the Skorohod space $\mathbb{D}(\R_{+},\R^{d})$. We say that $\widetilde{P}\in\mathcal{P}(\mathbb{D}(\R_{+},\R^{d}))$
is a solution to the martingale problem \ref{def:pb-02} with initial
condition $\widetilde{P}_{0}$ if, for all $\phi\in\mathcal{D}(L)$
and for all $t\geq0$,
\[
\phi(\xi_{t})-\phi(\xi_{0})-\int_{0}^{t}\int_{a,h\in\R^{d}}L\phi(\xi_{s})+(\phi(\xi_{s}+h)-\phi(\xi_{s}))\mu(\xi_{s},a,dh)\widetilde{P}_{s}(da)ds
\]
is a $\widetilde{P}$-martingale and the marginal of $\widetilde{P}$
at time $0$ is $\widetilde{P}_{0}$.

\end{definition}

In view of the above equation, the reason why the mass of $\mu(v,w,.)$
in zero (for any $v,w$) is not important is clear. According to Theorem
3.1 of \cite{graham-meleard-1997}, there exists a solution $\widetilde{P}$
of the problem \ref{def:pb-02} defined above (under hypothesis \ref{Hyp:01}).
We denote the marginal of $\widetilde{P}$ on $\mathbb{D}([0,T],\R^{d})$
by $\widetilde{P}_{0:T}$. We will work with this particular solution
in the following. This theorem also proves that (for all $q,t$):
\[
\Vert\mathcal{L}(\bZ_{0:t}^{1},\dots,\bZ_{0:t}^{q})-\mathcal{L}(\bZ_{0:t}^{1})^{\otimes q}\Vert_{TV}\leq2q(q-1)\frac{\Lambda t+\Lambda^{2}t^{2}}{N-1}\ ,
\]
 and 
\begin{equation}
\Vert\mathcal{L}(\bZ_{0:t}^{1})-\widetilde{P}_{0:t}\Vert_{TV}\leq6\frac{e^{\Lambda t}-1}{N+1}\ .\label{eq:approx-01}
\end{equation}

\begin{remark}If $\mu$ is fixed, there exists different $\widehat{\mu}$'s
having the proper marginal (that is, such that $\widehat{\mu}(.,.,.,\R^{d})=\mu(.,.,.)$).
In fact, it is the choice of $\widehat{\mu}$ that leads to having
different systems such as the Bird and Nanbu systems. We refer the
reader to \cite{graham-meleard-1994,graham-meleard-1999}, \cite{graham-meleard-1997}
p. 119-120 for very good discussions on the difference between the
Bird model and the Nanbu model. What matters here is that our result
applies to any system satisfying Hypothesis \ref{H1} and having jumps
of the form (\ref{Eq:BirdJump}).

\end{remark}

We can deduce propagation of chaos from the previous results, that
is for all $t$, for all $F$ bounded measurable, 
\[
\Vert\mathcal{L}(\bZ_{0:t}^{1},\dots,\bZ_{0:t}^{q})(F)-\widetilde{P}_{0:t}^{\otimes q}(F)\Vert_{TV\mbox{}}\leq\left(2q(q-1)\frac{\Lambda t+\Lambda^{2}t^{2}}{N-1}+6\frac{e^{\Lambda t}-1}{N+1}\right)\Vert F\Vert_{\infty}\ .
\]
 In Theorem \ref{Theo:boltzmann}, we will go further than the bound
in the equation above by writing an expansion of the left-hand side
term in powers of $N$ (see the discussion below Theorem \ref{Theo:boltzmann}
concerning the nature of this expansion). We will use techniques introduced
in \cite{graham-meleard-1997}. The main point is that one should
look at the processes backward in time.

\subsection{Statement of main results}

From now on, we will work with a fixed time horizon $T>0$ and a fixed
$q\in\N^{*}$.

\subsubsection{Expansion of the propagation of chaos}

We define for any $n,j\in\N^{*}$, $j\leq n$:
\[
\]
\[
[n]=\{1,2,\dots,n\}\,,\,\langle j,n\rangle=\{a:[j]\rightarrow[n],a\text{ injective }\}\ ,(n)_{j}=\#\langle j,n\rangle=\frac{n!}{(n-j)!}\ .
\]
We take $q\in\N^{*}$ and $T>0$. Let us set 
\[
\eta_{0:T}^{N}=\frac{1}{N}\sum_{1\leq i\leq N}\delta_{\bZ_{0:t}^{i}}\ ,\,(\eta_{0:T}^{N})^{\odot q}=\frac{1}{(N)_{q}}\sum_{a\in\langle q,N\rangle}\delta_{(\bZ_{0:T}^{a(1)},\dots,\bZ_{0:T}^{a(q)})}\ .
\]
 
\[
\]
 For any function $F:\mathbb{D}([0,T],\R^{d})^{q}\rightarrow\R$,
we call $(\eta_{0:T}^{N})^{\odot q}(F)$ a $U$-statistic. Note that
for all functions $F$, 
\begin{equation}
\E(F(\bZ_{0:T}^{1},\dots,\bZ_{0:T}^{q}))=\E((\eta_{0:T}^{N})^{\odot q}(F))\label{eq:defodot}
\end{equation}
because $(\bZ_{0:T}^{1},\dots,\bZ_{0:T}^{N})$ is exchangeable. We
define
\[
F_{\text{sym}}(x^{1},\dots,x^{q})=\frac{1}{q!}\sum_{\sigma\in\mathcal{S}_{q}}F(x^{\sigma(1)},\dots,x^{\sigma(q)})\ ,
\]
 where the sum is taken over the set $\mathcal{S}_{q}$ of the permutations
of $[q]$. We say that $F:\mathbb{D}([0,T],\R^{d})^{q}\rightarrow\R$
is symmetric if for all $\sigma$ in $\mathcal{S}_{q}$, $\forall x_{1},\dots,x_{q}\in\mathbb{D}([0,T],\R^{d})^{q}$,
$F(x_{\sigma(1)},\dots,x_{\sigma(q)})=F(x_{1},\dots,x_{q})$. If $F$
is symmetric then $F_{sym}=F$. Note that for all $F$, 

\[
(\eta_{0:T}^{N})^{\odot q}(F)=(\eta_{0:T}^{N})^{\odot q}(F_{\text{sym}})\,.
\]

\begin{theorem} \label{Theo:boltzmann} For all $q\geq1$, for any
bounded measurable symmetric $F$, for all $l_{0}\geq1$,

\begin{equation}
\E((\eta_{0:T}^{N})^{\odot q}(F))=\sum_{0\leq l\leq l_{0}}\left[\frac{1}{(N-1)^{l}}\Delta_{q,T}^{N,l}(F)\right]+\frac{1}{(N-1)^{l_{0}+1}}\overline{\Delta}_{q,T}^{N,l_{0}+1}(F)\label{eq:theo}
\end{equation}
 where the $\Delta_{q,T}^{N,l}$, $\overline{\Delta}_{q,T}^{N,l_{0}+1}$
are nonnegative measures uniformly bounded in $N$ (defined in Equations
(\ref{eq:def-delta}), (\ref{eq:def-delta-barre})). \end{theorem}

We will give a bound on these measures $\Delta$ and $\overline{\Delta}$
in (\ref{eq:borne-delta-delta-barre}). Let us define $\p_{T,q}^{N}(F)=\E((\eta_{T}^{N})^{\odot q}(F))$.
Regarding the fact that the theorem above is or is not a proper expansion,
what we can say is  that according to the terminology of \cite{del-moral-patras-rubenthaler-2008},
p. 782, we cannot say that the sequence of measure $(\p_{T,q}^{N})_{N\geq1}$
is differentiable up to any order because the $\Delta_{q,T}^{N,l}$
appearing in the development depend on $N$.

\subsubsection{Convergence results}

The main interest of Th. \ref{Theo:boltzmann} is that it gives us
sufficient knowledge of the particle system to prove an almost sure
convergence result and central-limit theorems. The key is to focus
on functions centered in the right way.

\begin{definition}\label{def:bsym}

We define a set of ``centered'' functions: 
\begin{eqnarray*}
 &  & \mathcal{B}_{0}^{sym}(q)=\Big\{ F:\mathbb{D}([0,T],\R^{qd})\rightarrow\R^{+},F\text{ measurable, symmetric, bounded},\\
 &  & ~~~~~~~~~~~~~~~~~~\int_{x_{1},\dots,x_{q}\in\mathbb{D}([0,T],E)}F(x_{1},\dots,x_{q})\widetilde{P}_{0:T}(dx_{q})=0\Big\}\ .
\end{eqnarray*}
 \end{definition}

We set (for $k$ even) 
\begin{equation}
J_{k}=\frac{k!}{2^{k/2}(k/2)!}\,.\label{eq:def-J-k}
\end{equation}
 (this is the number of partitions of $[k]$ into $k/2$ pairs).

\begin{proposition} \label{Prop:convq2} \label{Prop:conv-termes-centres}
(Proof in Subsection \ref{sub:Proof-of-Proposition-termes-centres})
For $q\geq1$, $F\in\mathcal{B}_{0}^{\text{sym}}(q)$, we have:
\begin{enumerate}
\item \label{point:prop-02}for $q$ odd, $N^{q/2}\E((\eta_{0:T}^{N})^{\odot q}(F))\underset{N\rightarrow+\infty}{\longrightarrow}0$, 
\item \label{point:prop-03}for $q$ even,
\begin{multline}
N^{q/2}\E((\eta_{0:T}^{N})^{\odot q}(F))\\
\convN\sum_{1\leq k\leq q/2}J_{q}\left(\begin{array}{c}
q/2\\
k
\end{array}\right)(-1)^{\frac{q}{2}-k}\\
\,\,\,\,\,\,\,\,\\
\times\E[\E_{\widetilde{\mathcal{K}}_{T}}(F(\widetilde{Z}_{0:T}^{1},\dots,\widetilde{Z}_{0:T}^{2k},\widetilde{\widetilde{Z}}_{0:T}^{2k+1},\dots,\widetilde{\widetilde{Z}}_{0:T}^{q})-F(\widetilde{\widetilde{Z}}_{0:T}^{1},\dots,\widetilde{\widetilde{Z}}_{0:T}^{q})|\widetilde{L}_{1,q})\\
\,\,\,\,\,\,\,\times\prod_{1\leq i\leq q/2}\int_{0}^{T}\Lambda\widetilde{K}_{s}^{2i-1}\widetilde{K}_{s}^{2i}ds]\label{eq:conWick01}
\end{multline}
 where the limit, indeed, does not depend on $N$ (the notations $\widetilde{\widetilde{Z}}$,
$\widetilde{\widetilde{Z}}$, $\widetilde{L}$, $\widetilde{K}$,
$\E_{\widetilde{\mathcal{K}}_{T}}$ will be introduced in Subsection
\ref{sub:Auxiliary-systems}). This limit takes a particular form
if $F=(f_{1}\otimes\dots\otimes f_{q})_{\mbox{sym }}$ (with $f_{i}\in\mathcal{B}_{0}^{sym}(1)$,
$\forall i$) (see Corollary \ref{Cor:wickproduct}).
\end{enumerate}
\end{proposition}

 Using the above Proposition, some combinatorics and Borel-Cantelli
Lemma, we prove the following theorem (see the proof in Subsection
\ref{sub:Almost-sure-convergence}).

\begin{theorem}\label{Th:conv-ps-debut}For any measurable bounded
$f$, $T\geq0$, 
\[
\eta_{0:T}^{N}(f)\underset{N\rightarrow+\infty}{\overset{\text{a.s.}}{\longrightarrow}}\widetilde{P}_{0:T}(f)\ .
\]

\end{theorem}

Using the above results and a computation on characteristic functions,
we then prove the following theorem (see the proof in Subsection \ref{sub:Central-limit-theorem}).

\begin{theorem}\label{Th:TCL-debut}For all $f_{1},\dots,f_{q}\in\mathcal{B}_{0}^{sym}(1)$,
for all $T\geq0$,
\[
N^{q/2}(\eta_{0:T}^{N}(f_{1}),\dots,\eta_{0:T}^{N}(f_{q}))\underset{N\rightarrow+\infty}{\overset{\text{law}}{\longrightarrow}}\mathcal{N}(0,K)\ ,
\]

(the matrix $K$ is given in (\ref{eq:def-K})).

\end{theorem}

A similar result can be found in \cite{uchiyama-1983-a}, under the
assumption that the initial law $\widetilde{P}_{0}$ falls in some
particular set. This assumption makes it difficult to compare our
covariance $K$ to the ones found in \cite{uchiyama-1983-a} (the
expressions of $K$ varies according to the subset $\widetilde{P}_{0}$
is in). A similar result can also be found in \cite{uchiyama-1988},
this time for particles moving in a set which can only be countable. 

The result in \cite{shiga-tanaka-1985} has common points with the
theorem above, but the kernel $Q^{(n)}$ defined in \cite{shiga-tanaka-1985}
is asymmetric. The variance appearing in Th. 2.1 of \cite{shiga-tanaka-1985}
(Equation (2.7)) could be expressed as an expectation over random
trees (if one uses Equation (2.22) of \cite{shiga-tanaka-1985}) but
the asymmetry of the kernel would make it different from our $K$
anyway. The fact that we do not need an assumption of the kind of
(2.4) p. 443 of \cite{shiga-tanaka-1985} is another difference between
our result and Th. 2.1 of \cite{shiga-tanaka-1985}. \cite{dawson-zheng-1991}
extends the result of \cite{shiga-tanaka-1985} to a case where the
jump rate is not bounded (without the second part of our Hypothesis
\ref{H1}) but it is limited to processes in $\mathbb{Z}_{+}$. %

Using classical techniques, we obtain the following Corollary.

\begin{corollary}\label{cor:TCL}For any $q\in\N^{*}$, $F$ bounded
measurable and symmetric, we have
\[
\sqrt{N}\left((\eta_{0:T}^{N})^{\odot q}(F)-\widetilde{P}_{0:T}(F)\right)\convloi\mathcal{N}(0,q^{2}(\widetilde{P}_{0:T}((F^{(1)})^{2})+V_{0:T}((F^{(1)})^{2}))\,,
\]
where $F^{(1)}(x_{1})=\int_{\mathbb{D}([0,T],\R^{d})^{q-1}}F(x_{1},\dots,x_{q})\widetilde{P}_{0:T}(dx_{2},\dots,dx_{q})$
and $V_{0:T}$ is defined in (\ref{eq:def-V-0-T}). 

\end{corollary}

\section{Other systems of particles\label{sec:Other-systems-of}}

In this section, we introduce the particle systems that we will need
for the proofs of the main results.

\subsection{Backward point of view\label{sub:Backward-point-of}}

 For $\lambda>0$, we denote by $\mathcal{E}(\lambda)$ the exponential
law of parameter $\lambda$. For any $x\in\R$, we define $\lfloor x\rfloor:=\sup\{i\in\Z,i\leq x\}$,
$\lceil x\rceil=\inf\{i\in\Z,i\geq x\}$.

We intend to construct a system of particles $(Z_{0:T}^{i})_{1\leq i\leq N}$
such that the first $q$ particles have the same law as $(\overline{Z}_{0:T}^{1},\dots,\overline{Z}_{0:T}^{q})$
(see Lemma \ref{Lem:equlaw1}). We use the fact that the processes
$(N_{i,j}(T-t))_{0\leq t\leq T}$ are Poisson processes to construct
the interaction graph for the first $q$ particles moving backward
in time. The system of particles $(Z_{0:T}^{i})_{1\leq i\leq q}$
is indeed the central system in our paper, hence all other systems
will be compared to it.

We start at $s=0$ with $C_{0}^{i}=\{i\}$, for all $i\in[q]$. For
$i\in[q]$, we want to define $(C_{s}^{i})_{s\geq0}$, $(K_{s}^{i})_{s\geq0}$
(respectively taking values in $\mathcal{P}(\N)$, $\N^{*}$). We
take $(U_{k})_{1\leq k},(V_{k})_{1\leq k}$ i.i.d. $\sim\mathcal{E}(1)$.
In all the following, we will use the conventions: $\inf\emptyset=+\infty$
and $(\dots)_{+}$ is the nonnegative part. The processes $(C^{i}),\,(K^{i})$
are piecewise constant and make jumps. At any time $t$, we set $K_{t}^{i}=\#C_{t}^{i}$.
For all $t\in[0,T]$, we set $K_{t}=\#\left(C_{t}^{1}\cup\dots\cup C_{t}^{q}\right)$.

Before considering the technical details, let us explain our purpose.
The population $C^{1}\cup\dots\cup C^{q}$ is allowed to get a new
particle from $[N]\backslash(C^{1}\cup\dots\cup C^{q})$ by growing
a link to this particle (this particle is chosen uniformly). If such
an event happens at time $t_{0}$, then the waiting time until the
next such event is of law $\mathcal{E}\left(\frac{\Lambda K_{t_{0}}(N-K_{t_{0}})_{+}}{N-1}\right)$.
To put it briefly, we say that this kind of event happens at a rate
$\frac{\Lambda K_{.}(N-K_{.})_{+}}{N-1}$. The population $C^{1}\cup\dots\cup C^{q}$
is allowed to form links between particles of $C^{1}\cup\dots\cup C^{q}$
(we will call ``loops'' these particular links), and this kind of
event happens at a rate $\frac{\Lambda K_{.}(K_{.}-1)}{2(N-1)}$ (the
newly linked particles are chosen uniformly). The processes $C^{1}$,
\ldots, $C^{q}$ are used in Definition \ref{Def:bZ} to define the
process $(Z^{i})_{1\leq i\leq N}$. As will be seen below, the link
times correspond to the interaction times of some particles.

We define the jump times recursively by $T_{0}=0$ and:

\begin{eqnarray*}
T'_{k} & = & \inf\left\{ T_{k-1}\leq s\leq T:(s-T_{k-1})\times\frac{\Lambda K_{T_{k-1}}(N-K_{T_{k-1}})_{+}}{N-1}\geq U_{k}\right\} \\
T''_{k} & = & \inf\left\{ T_{k-1}\leq s\leq T:(s-T_{k-1})\times\frac{\Lambda K_{T_{k-1}}(K_{T_{k-1}}-1)}{2(N-1)}\geq V_{k}\right\} \\
T_{k} & = & \inf(T'_{k},T''_{k})\ .
\end{eqnarray*}
Here, we use a representation with $\inf$'s to emphasize the fact
that these jump times are the jump times of Poisson processes with
certain intensities. At $T_{k}$:
\begin{itemize}
\item If $T_{k}=T'_{k}$, we draw 
\begin{equation}
r(k)\mbox{ uniformly in }C_{T_{k}-}^{1}\cup\dots\cup C_{T_{k}-}^{q}\,,\,j(k)\mbox{ uniformly in }[N]\backslash(C_{T_{k}-}^{1}\cup\dots\cup C_{T_{k}-}^{q})\,.\label{eq:tirage-r-j-01}
\end{equation}
 For any $i$ such that $r(k)\in C_{T_{k}-}^{i}$, we then perform
the jump: $C_{T_{k}}^{i}=C_{T_{k}-}^{i}\cup\{j(k)\}$.

Note that the $(\dots)_{+}$ in the definition of $T'_{k}$ above
prevents us from being in the situation where we would be looking
for $j(k)$ in $\emptyset$.

\item If $T_{k}=T_{k}''$, we draw 
\begin{equation}
r(k)\mbox{ uniformly in }C_{T_{k}-}^{1}\cup\dots\cup C_{T_{k}-}^{q}\,,\,j(k)\mbox{ uniformly in }C_{T_{k}-}^{1}\cup\dots\cup C_{T_{k}-}^{q}\backslash\{r(k)\}\,.\label{eq:tirage-r-j-02}
\end{equation}

\end{itemize}
For each $k$ such that $T_{k}\leq T$, if $l_{1},l_{2}$ are such
that $r(k)\in C_{T_{k}-}^{l_{1}}$, $j(k)\in C_{T_{k}-}^{l_{2}}$,
we say that there is a link between $C^{l_{1}}$ and $C^{l_{2}}$.
This whole construction is analogous to the construction of the interaction
graph found in \cite{graham-meleard-1997}, p. 122. For all $t\leq T$,
we set 
\[
\mathcal{K}_{t}=\left(K_{s}^{i}\right)_{1\leq i\leq q,0\leq s\leq t}\,.
\]

Let us now define an auxiliary process $({Z}_{s})_{0\leq s\leq T}=({Z}_{s}^{i})_{0\leq s\leq T,1\leq i\leq N}$
of $N$ particles in $\R^{d}$.

\begin{definition}\label{Def:bZ}Let $k'=\sup\left\{ k,T_{k}<\infty\right\} $.
The interaction times of $(Z_{s}^{i})_{1\leq s\leq T,1\leq i\leq N}$
are $ $$T-T_{k'}\leq T-T_{k'-1}\leq\dots\leq T-T_{1}$. (We say that
the interaction times are defined backward in time.)
\begin{itemize}
\item $Z_{0}^{1},\dots,Z_{0}^{N}$ are i.i.d. $\sim\widetilde{P}_{0}$ 
\item Between the times $(T-T_{k})_{k\geq1}$, the $Z^{i}$'s evolve independently
of each other according to the Markov generator $L$. 
\item At a  time $T-T_{k}$, $(Z^{i})_{1\leq i\leq N}$ undergoes an interaction
that has the same law as in Definition \ref{Def:system}, with $(i,j)$
replaced by $(r(k),j(k))$.
\end{itemize}
It is worth noting that for $N$ large and $i\notin[q]$, it is very
likely that the particle $i$ has no interaction with the other particles.

\end{definition}

For all $0\leq t\leq T$, we set 
\[
L_{t}=\#\{k\in\N:T_{k}\leq t,T_{k}=T''_{k}\}\ .
\]

\begin{example} Take $q=2$. Suppose for example, that $T_{0}=0$,
$T_{1}=T/2$, $T_{2}=3T/4$, $T_{3}=+\infty$, $r(1)=1$, $j(1)=2$,
$r(2)=2$, $j(2)=3$.

Then
\begin{itemize}
\item for $s\in[0,T/2[$, $K_{s}=2$, $L_{s}=0$, $K_{s}^{1}=K_{s}^{2}=1$, 
\item for $s\in[T/2,3T/4[$, $K_{s}=2$, $L_{s}=1$, $K_{s}^{1}=K_{s}^{2}=1$, 
\item for $s\in[3T/4,T]$, $K_{s}=3$, $L_{s}=1$, $K_{s}^{1}=1$, $K_{s}^{2}=2$
. 
\end{itemize}
\end{example}

\begin{figure}
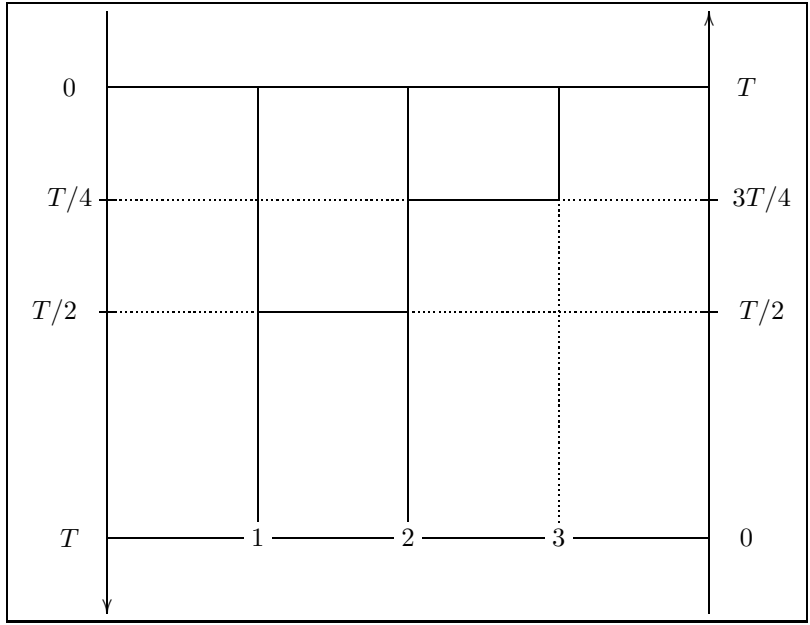

\caption{Interaction graph for $(Z_{0:T}^{1},Z_{0:T}^{2})$}

\label{fig:01}

\[\fbox{
\xygraph{!{<0cm,0cm>;<1cm,0cm>:<0cm,1cm>::}!{(0,0)}*{}="00"!{(0,6)}*{}="06"!{(0,-1)}="0-1"!{(0,7)}="07"!{(2,0)}{1}="20"!{(2,6)}*{}="26"!{(2,3)}*{}="22"!{(4,0)}{2}="40"!{(4,6)}*{}="46"!{(4,3)}*{}="42"!{(4,4.5)}*{}="44"!{(6,0)}{3}="60"!{(6,6)}*{}="66"!{(6,4.5)}*{}="64"!{(8,0)}*{}="80"!{(8,6)}*{}="86"!{(8,-1)}="8-1"!{(8,7)}="87"!{(-0.5,0)}{T}!{(-0.5,6)}{0}!{(0.1,4.5)}*{}="g3"!{(-0.1,4.5)}="d3"!{(-0.7,3)}{T/2}!{(0.1,3)}*{}="g4"!{(-0.1,3)}="d4"!{(-0.5,4.5)}{T/4}!{(8.5,0)}{0}!{(8.5,6)}{T}!{(8.7,3)}{T/2}!{(8.1,4.5)}*{}="g1"!{(7.9,4.5)}="d1"!{(8.7,4.5)}{3T/4}!{(8.1,3)}*{}="g2"!{(7.9,3)}="d2""00"-"06""20"-"26""40"-"46""60"-@{.}"64""64"-"66""80"-"86""0-1"-@{<-}"07""8-1"-@{->}"87""06"-"86""44"-"64""22"-"42""42"-@{.}"g2""64"-@{.}"g1""00"-"20""20"-"40""40"-"60""60"-"80""g1"-"d1""g2"-"d2""g3"-"d3""g4"-"d4""d3"-@{.}"44""d4"-@{.}"22"} 
}\]
\end{figure}

Figure \ref{fig:01} is a pictorial representation of the example
above. The time arrow for the particle is on the left. The time arrow
for the processes $(C^{i})$, $(K^{i})$ is on the right. What we
draw here is called the graph of interactions (for $Z_{0:T}^{1}$,
$Z_{0:T}^{2}$) in \cite{graham-meleard-1997,graham-meleard-1999}.
Suppose we want to simulate $Z_{0:T}^{1}$, $Z_{0:T}^{2}$. We first
simulate the interaction times of the system. Suppose that these are
exactly $T-T_{1}$, $T-T_{2}$ with $T_{1}$, $T_{2}$ coming from
the example above. In Figure \ref{fig:01}, solid vertical lines represent
the trajectories we have to simulate to obtain $Z_{0:T}^{1}$, $Z_{0:T}^{2}$.
The particle numbers are to be found at the bottom of the graph. The
horizontal solid lines stand for the interaction we have to simulate
in order to obtain $Z_{0:T}^{1}$, $Z_{0:T}^{2}$ (they may or may
not induce jumps for the particles). For example, a horizontal solid
line between the vertical solid lines representing the trajectories
of $Z_{0:T}^{1}$, $Z_{0:T}^{2}$ stands for an interaction between
particle $1$ and particle $2$. The interactions are simulated following
Definition \ref{Def:bZ}. The trajectory $Z_{0:T}^{3}$ is represented
by a solid line between the times $0$ and $T/4$ and by a dashed
line between the times $T/4$ and $T$, with the number $3$ at the
bottom. As we want to simulate $Z_{0:T}^{1}$ and $Z_{0:T}^{2}$ and
we have simulated the jumps as in the example above, then we are not
interested in $Z_{t}^{3}$ for $t>T/4$ and we are not interested
in any $Z_{0:T}^{i}$ with $i\geq4$. Again, the time for the particles
should be read on the left.

The following lemma should be kept in mind throughout the whole paper.

\begin{lemma} \label{Lem:fondamental}Let us denote an inhomogeneous
Poisson process of rate $(\lambda_{t})_{t\geq0}$ by $(N_{t}^{\lambda})_{t\geq0}$
($\lambda$ is supposed to be piecewise constant).
\begin{enumerate}
\item \label{lem-point-01}Let us denote the jump times of $N^{\lambda}$
by $\tau_{1}<\tau_{2}<\dots$. Then for all $k\in\N^{*}$, 
\[
\mathcal{L}(\tau_{1},\tau_{2},\dots,\tau_{k}|\tau_{k}\leq T<\tau_{k+1})
\]
 is the law of the order statistics of $k$ independent variables
of law of density $t\mapsto\lambda_{t}/\int_{0}^{T}\lambda_{s}ds$
on $[0,T]$. 
\item \label{lem-point-02}For any $j\in\N^{*}$, take piecewise constant
processes $(\alpha_{t}^{j})_{t\geq0}$ such that for all $t$, $0\leq\alpha_{t}^{1}\leq\alpha_{t}^{1}+\alpha_{t}^{2}\leq\dots\leq\alpha_{t}^{1}+\dots+\alpha_{t}^{j}\leq1$.
Suppose we take $(W_{k})_{k\geq0}$ i.i.d. random variables of uniform
law on $[0,1]$ independent of $N^{\lambda}$. Suppose we set $j$
processes $(N_{t}^{i})_{t\geq0,1\leq i\leq j}$ such that $N_{0}^{i}=0$
for all $i$, the processes $N^{i}$'s are a.s. piecewise constant
and may jump at the jump times of $N^{\lambda}$ following this rule:
$\Delta N_{t}^{i}=1$ if and only if $\Delta N_{t}^{\lambda}=1$ and
$\alpha_{t}^{1}+\dots+\alpha_{t}^{i-1}\leq W_{N_{t-}^{\lambda}}<\alpha_{t}^{1}+\dots+\alpha_{t}^{i}$.
Then the $N^{i}$'s are $j$ independent inhomogeneous Poisson processes
such that $N^{i}$ has rate $(\alpha_{t}^{i}\times\lambda_{t})_{t\geq0}$
for all $i$. 
\item \label{lem-point-03}Take $j\in\N^{*}$. Take $j$ independent inhomogeneous
Poisson processes $(N^{i})_{1\leq i\leq j}$ respectively of rate
$(\lambda_{t}^{i})_{t\geq0}$ (the $\lambda^{i}$ are piecewise constant
processes). Then $N_{t}=N_{t}^{1}+\dots+N_{t}^{j}$ is an inhomogeneous
Poisson process of rate $(\lambda_{t}^{1}+\dots+\lambda_{t}^{j})_{t\geq0}$
and for all $i$, $s$, $\p(\Delta N_{t}^{i}=\Delta N_{t}|\Delta N_{t}=1)=\p(\Delta N_{t}^{i}=\Delta N_{t}|\Delta N_{t}=1,(N_{s}^{k})_{1\leq k\leq j,0\leq s<t})=\frac{\lambda_{t}^{i}}{\lambda_{t}^{1}+\dots+\lambda_{t}^{j}}$. 
\end{enumerate}
\end{lemma}

The point \ref{lem-point-01} of the above Lemma is derived from the
Mapping Theorem of \cite{kingman-1993} applied to Cox processes (see
\cite{kingman-1993}, p. 17 and 71). The point \ref{lem-point-02}
of the above Lemma is derived from the Coloring Theorem of \cite{kingman-1993}
(see p. 53). The point \ref{lem-point-03} of the above Lemma is derived
from the Superposition Theorem of \cite{kingman-1993} (see p.16).

We then obtain the following Lemma.

\begin{lemma} \label{Lem:equlaw1} For all $T\geq0$, $(Z_{0:T}^{1},\dots,Z_{0:T}^{q})\eqlaw(\bZ_{0:T}^{1},\dots,\bZ_{0:T}^{q})$.
\end{lemma}

The proof can be found in Section \ref{sub:Proof-of-Lemma}.

\subsection{Auxiliary systems\label{sub:Auxiliary-systems}}

We now define an auxiliary system $(\widetilde{Z}_{0:T}^{i})_{i\geq1}$
with an infinite number of particles. We start at $s=0$ with $\widetilde{C}_{0}^{i}=\{i\}$,
for all $i\in[q]$. For $1\leq i\leq N$, we define $(\widetilde{C}_{s}^{i})_{s\geq0,1\leq i\leq q}$,
$(\widetilde{K}_{s}^{i})_{s\geq0,1\leq i\leq q}$ (respectively taking
values in $\mathcal{P}(\N),\N$) by the following. The processes $(\widetilde{C}^{i}),(\widetilde{K}^{i})$
are piecewise constant. At any time $t$, $\widetilde{K}_{t}^{i}=\#\widetilde{C}_{t}^{i}$. 

Before going into the technical details, let us explain the purpose
of the construction. We intend to build a process $\widetilde{C}=\widetilde{C}^{1}\cup\dots\cup\widetilde{C}^{q}$
that grows at a rate $\Lambda\widetilde{K}_{.}$ (by creating links
to new particles) and that forms links between two particles of $\widetilde{C}^{1}\cup\dots\cup\widetilde{C}^{q}$
at a rate $\frac{\Lambda\widetilde{K}_{.}(\widetilde{K}_{.}-1)}{2(N-1)}$
(in this case, the links will also be called loops). Note that for
$k\in\N^{*}$, $\frac{\Lambda k(N-k)_{+}}{N-1}\underset{N\rightarrow+\infty}{\sim}\Lambda k$,
and that, for all $N$, $\Lambda k\geq\frac{\Lambda k(N-k)_{+}}{N-1}$.
Given the processes $C^{1}$, ..., $C^{q}$, we add to them jumps
(and elements) so as to form populations $\widetilde{C}^{1}$, ...,
$\widetilde{C}^{q}$ with the desired growth rate. It will be easier
to write inequalities if the populations $C^{i}$ and $\widetilde{C}^{i}$
are coupled (see for example the proof of Theorem \ref{Theo:boltzmann}).
That is why we use the jump times $T_{k}$, $T'_{k}$, $T''_{k}$
in the construction below. 

We take $(\widetilde{U}_{k})_{k\geq1}$, $(\widetilde{U}_{k}')_{k\geq1}$
i.i.d. $\sim\mathcal{E}(1)$. We define the jump times recursively
by $\widetilde{T}_{0}=0$ and

\begin{eqnarray*}
\widetilde{T}'_{k} & = & \inf\left\{ \widetilde{T}_{k-1}\leq s\leq T,\,(s-\widetilde{T}_{k-1})\left(\Lambda\tK_{\widetilde{T}_{k-1}}-\frac{\Lambda K_{\widetilde{T}_{k-1}}(N-K_{\widetilde{T}_{k-1}})_{+}}{N-1}\right)\geq\widetilde{U}_{k}\right\} \\
\widetilde{T}_{k}'' & = & \inf\left\{ \widetilde{T}_{k-1}\leq s\leq T,\,(s-\widetilde{T}_{k-1})\times\frac{\Lambda\widetilde{K}_{\widetilde{T}_{k-1}}(\widetilde{K}_{\widetilde{T}_{k-1}}-1)-\Lambda K_{\widetilde{T}_{k-1}}(K_{\widetilde{T}_{k-1}}-1)}{2(N-1)}\geq\widetilde{U}_{k}'\right\} \\
\widetilde{T}_{k} & = & \inf(\widetilde{T}'_{k},\widetilde{T}_{k}'',\inf\{T_{l}:T_{l}>\widetilde{T}_{k-1}\})
\end{eqnarray*}
(recall the definition of the process $(K_{t})$ and the $T_{k}$'s
from Subsection \ref{sub:Backward-point-of}). Note that $\{T_{k},k\geq0\}\subset\{\tT_{k},k\geq0\}$.
At $\widetilde{T}_{k}$:
\begin{itemize}
\item If $\widetilde{T}_{k}=\widetilde{T}'_{k}$ (note that it implies that
$\widetilde{K}_{\widetilde{T}_{k}-}-\frac{K_{\widetilde{T}_{k}-}(N-K_{\widetilde{T}_{k}-})_{+}}{N-1}>0$):

\begin{itemize}
\item With probability 
\begin{equation}
\frac{K_{\widetilde{T}_{k}-}-\frac{K_{\widetilde{T}_{k}-}(N-K_{\widetilde{T}_{k}-})_{+}}{N-1}}{\widetilde{K}_{\widetilde{T}_{k}-}-\frac{K_{\widetilde{T}_{k}-}(N-K_{\widetilde{T}_{k}-})_{+}}{N-1}}\,,\label{eq:proba-choix-01}
\end{equation}
 we draw $\widetilde{r}(k)$ uniformly in $C_{\widetilde{T}_{k}-}^{1}\cup\dots\cup C_{\widetilde{T}_{k}-}^{q}$
and 
\[
\widetilde{j}(k)=\min\left\{ \N^{*}\backslash(\widetilde{C}_{\widetilde{T}_{k}-}^{1}\cup\dots\cup\widetilde{C}_{\widetilde{T}_{k}-}^{q}\cup[N])\right\} \,.
\]
 The following jump is performed: for all $l$ such that $\widetilde{r}(k)\in\widetilde{C}_{\widetilde{T}_{k}-}^{l}$,
$\widetilde{C}_{\widetilde{T}_{k}}^{l}=\widetilde{C}_{\widetilde{T}_{k}-}^{l}\cup\left\{ \widetilde{j}(k)\right\} $. 
\item With probability 
\begin{equation}
\frac{\widetilde{K}_{\widetilde{T}_{k}}-K_{\widetilde{T}_{k}-}}{\widetilde{K}_{\widetilde{T}_{k}-}-\frac{K_{\widetilde{T}_{k}-}(N-K_{\widetilde{T}_{k}-})_{+}}{N-1}}\,,\label{eq:proba-choix-02}
\end{equation}
we draw $\widetilde{r}(k)$ uniformly in $(\widetilde{C}_{\widetilde{T}_{k}-}^{1}\backslash C_{\widetilde{T}_{k}-}^{1})\cup\dots\cup(\widetilde{C}_{\widetilde{T}_{k}-}^{q}\backslash C_{\widetilde{T}_{k}-}^{q})$
and\\
 $\widetilde{j}(k)=\min\left\{ \N^{*}\backslash(\widetilde{C}_{\widetilde{T}_{k}-}^{1}\cup\dots\cup\widetilde{C}_{\widetilde{T}_{k}-}^{q}\cup[N])\right\} $.
The following jump is performed: for all $l$ such that $\widetilde{r}(k)\in\widetilde{C}_{\widetilde{T}_{k}-}^{l}$,
$\widetilde{C}_{\widetilde{T}_{k}}^{l}=\widetilde{C}_{\widetilde{T}_{k}-}^{l}\cup\left\{ \widetilde{j}(k)\right\} $. 
\end{itemize}
\item If $\widetilde{T}_{k}=\widetilde{T}''_{k}$ (note that it implies
$\widetilde{K}_{\widetilde{T}_{k}-}>K_{\widetilde{T}_{k}-}$):

\begin{itemize}
\item With probability
\begin{equation}
\frac{(\widetilde{K}_{\widetilde{T}_{k}-}-K_{\widetilde{T}_{k}-})K_{\widetilde{T}_{k}-}+(\widetilde{K}_{\widetilde{T}_{k}-}-K_{\widetilde{T}_{k}-})(\widetilde{K}_{\widetilde{T}_{k}-}-K_{\widetilde{T}_{k}-}-1)}{\widetilde{K}_{\widetilde{T}_{k}-}(\widetilde{K}_{\widetilde{T}_{k}-}-1)-K_{\widetilde{T}_{k}-}(K_{\widetilde{T}_{k}-}-1)}\,,\label{eq:proba-saut-en-plus}
\end{equation}
 we draw $\widetilde{r}(k)$ uniformly in $(\widetilde{C}_{\widetilde{T}_{k}-}^{1}\backslash C_{\widetilde{T}_{k}-}^{1})\cup\dots\cup(\widetilde{C}_{\widetilde{T}_{k}-}^{q}\backslash C_{\widetilde{T}_{k}-}^{q})$
and $\widetilde{j}(k)$ uniformly in $(\widetilde{C}_{\widetilde{T}_{k}-}^{1}\cup\dots\cup\widetilde{C}_{\widetilde{T}_{k}-}^{q})\backslash\left\{ \widetilde{r}(k)\right\} $. 
\item With probability 
\begin{equation}
\frac{(\widetilde{K}_{\widetilde{T}_{k}-}-K_{\widetilde{T}_{k}-})K_{\widetilde{T}_{k}-}}{\widetilde{K}_{\widetilde{T}_{k}-}(\widetilde{K}_{\widetilde{T}_{k}-}-1)-K_{\widetilde{T}_{k}-}(K_{\widetilde{T}_{k}-}-1)}\,,\label{eq:proba-saut-en-plus-02}
\end{equation}
 we draw $\widetilde{r}(k)$ uniformly in $(\widetilde{C}_{\widetilde{T}_{k}-}^{1}\backslash C_{\widetilde{T}_{k}-}^{1})\cup\dots\cup(\widetilde{C}_{\widetilde{T}_{k}-}^{q}\backslash C_{\widetilde{T}_{k}-}^{q})$
and $\widetilde{j}(k)$ uniformly in $C_{\widetilde{T}_{k}-}^{1}\cup\dots\cup C_{\widetilde{T}_{k}-}^{q}$. 
\end{itemize}
\item If $\widetilde{T}_{k}=T'_{l}$ for some $l$, we take $\widetilde{r}(k)=r(l),\widetilde{j}(k)=j(l)$
as in  (\ref{eq:tirage-r-j-01}). The following jump is performed:
$\widetilde{C}_{\tilde{T}_{k}}^{l}=\widetilde{C}_{\tilde{T}_{k}-}^{l}\cup\{\widetilde{j}(k)\}$
. 
\item If $\widetilde{T}_{k}=T''_{l}$ for some $l$, we take $\widetilde{r}(k)=r(l),\,\widetilde{j}(k)=j(l)$
as in (\ref{eq:tirage-r-j-02}). 
\end{itemize}
We define 
\begin{equation}
\widetilde{L}_{t}=\#\left\{ k:\widetilde{T}_{k}\in\left\{ T''_{l},\widetilde{T}_{l}'',l\geq1\right\} ,\widetilde{T}_{k}\leq T\right\} \,.\label{eq:def-L-T}
\end{equation}
 We set for all $s,t\leq T,\,i\in[q]$, 
\[
\widetilde{K}_{s}=\#(\widetilde{C}_{s}^{1}+\dots+\widetilde{C}_{s}^{q})\ ,
\]
 
\[
\widetilde{\mathcal{K}}_{t}=(\widetilde{K}_{s}^{j})_{1\leq j\leq q,0\leq s\leq t\,.}
\]
 \begin{notation}

We set $\E_{\mathcal{K}_{t}}(\dots)=\E(\dots|\mathcal{K}_{t})$ ,
$\p_{\mathcal{K}_{t}}(\dots)=\p(\dots|\mathcal{K}_{t})$, $\E_{\widetilde{\mathcal{K}}_{t}}(\dots)=\E(\dots|\widetilde{\mathcal{K}}_{t})$,
$\p_{\widetilde{\mathcal{K}}_{t}}(\dots)=\p(\dots|\widetilde{\mathcal{K}}_{t})$,
$\E_{\mathcal{K}_{t},\widetilde{\mathcal{K}}_{t}}(\dots)=\E(\dots|\mathcal{K}_{t},\widetilde{\mathcal{K}}_{t})$. 

\end{notation}

For all $k$, $l_{1},l_{2}$ such that $\widetilde{r}(k)\in\widetilde{C}_{\widetilde{T}_{k}-}^{l_{1}}$,
$\widetilde{j}(k)\in\widetilde{C}_{\widetilde{T}_{k}-}^{l_{2}}$,
we say that there is a link between $\widetilde{C}^{l_{1}}$ and $\widetilde{C}^{l_{2}}$.
We define \label{We-define}for all $I$ subset of $[q]$, 
\begin{multline}
\mathcal{T}^{I}=\left\{ \widetilde{T}_{k}\leq T,k\geq0,\,\widetilde{T}_{k}\in\left\{ T'_{l},\widetilde{T}'_{l},l\geq1\right\} ,\widetilde{r}(k)\mbox{ or }\widetilde{j}(k)\in\cup_{i\in I}\widetilde{C}_{\widetilde{T}_{k}-}^{i}\right\} \\
\cup\left\{ \widetilde{T}_{k}\leq T,k\geq0,\,\widetilde{T}_{k}\in\left\{ T''_{l},\widetilde{T}''_{l},l\geq1\right\} ,\widetilde{r}(k)\mbox{ and }\widetilde{j}(k)\in\cup_{i\in I}\widetilde{C}_{\widetilde{T}_{k}-}^{i}\right\} \,.\label{eq:def-Tronde}
\end{multline}
 In other words, $\mathcal{T}^{I}$ is the set of jump times $\widetilde{T}_{k}$
such that 
\begin{itemize}
\item if $\Delta\widetilde{L}_{\widetilde{T}_{k}}=0$, $\exists l\in I$
such that $\widetilde{C}^{l}$ jumps in $\widetilde{T}_{k}$,
\item if $\Delta\widetilde{L}_{\widetilde{T}_{k}}\neq0$, $\exists l_{1},l_{2}\in I$
such that $\widetilde{r}(k)\in\widetilde{C}_{\widetilde{T}_{k}-}^{l_{1}}$,
$\widetilde{j}(k)\in\widetilde{C}_{\widetilde{T}_{k}-}^{l_{2}}$.
\end{itemize}
We define for all $t,j$,

\begin{equation}
\widetilde{L}_{t}^{I}=\#\left\{ s\in\mathcal{T}^{I}\,,\,s\leq t\,,\,s\in\left\{ T''_{l},\widetilde{T}''_{l},l\geq1\right\} \right\} \,.\label{eq:def-L-tilde-(j-1,j)}
\end{equation}
We have $\widetilde{L}_{t}=\widetilde{L}_{t}^{[q]}$. The following
lemma is a consequence of Lemma \ref{Lem:fondamental}. Its proof
is elementary but quite long. It is written in Section \ref{sec:Proof-of-Lemma}.

\begin{lemma} \label{Lem:law-K-tilde}
\begin{enumerate}
\item \label{loiKtilde:02}The process $(\widetilde{K}_{s})_{0\leq s\leq T}$
is piecewise constant, has jumps of size $1$ and satisfies, for all
$s,t$ such that $0\leq s\leq t$, 
\[
\p(\widetilde{K}_{t}=\widetilde{K}_{s}|\tK_{s})=\exp(-\Lambda(t-s)\widetilde{K}_{s})\ .
\]
 The process $(\widetilde{L}_{s})_{0\leq s\leq T}$ is piecewise constant,
has jumps of size $1$ and satisfies for all $s,t$ such that $0\leq s\leq t$
\[
\p(\widetilde{L}_{t}=\widetilde{L}_{s}|\widetilde{L}_{s},(\widetilde{K}_{u})_{0\leq u\leq t})=\exp\left(-\int_{s}^{t}\frac{\Lambda\widetilde{K}_{u}(\widetilde{K}_{u}-1)}{N-1}du\right)\,.
\]
 Hence, knowing $(\widetilde{K}_{u})_{0\leq u\leq T}$, $(\widetilde{L}_{t})_{t\geq0}$
is an inhomogeneous Poisson process of rate $(\Lambda\widetilde{K}_{u}(\widetilde{K}_{u}-1)/(N-1))_{0\leq u\leq T}$. 
\item \label{lemKtilde:03}For all $t$, $\widetilde{K}_{t}\geq K_{t}$
and $\widetilde{L}_{t}\geq L_{t}$ a.s. . 
\item \label{lemKtilde:04}If $T_{1}=\tT_{1}$, \dots, $T_{k}=\tT_{k}$
then $\tK_{T_{k}}=K_{T_{k}}$, $\widetilde{L}_{T_{k}}=L_{T_{k}}$. 
\item \label{lemKtilde:01}The processes $(\widetilde{K}_{t}^{i})_{0\leq s\leq T}$
are independent. They are piecewise constant, have jumps of size $1$
and satisfy for all $s,t$ such that $0\leq s\leq t$, for all $i\in[q]$,
\[
\p(\widetilde{K}_{t}^{i}=\widetilde{K}_{s}^{i}|\widetilde{K}_{s}^{i})=\exp(-\Lambda(t-s)\widetilde{K}_{s}^{i})\ .
\]
 These processes are thus $q$ independent Yule processes (see \cite{athreya-ney-2004},
p. 102-109, p. 109 for the law of the Yule process). We have for all
$k\geq q$, 
\begin{equation}
\p(\widetilde{K}_{T}=k)=\left(\begin{array}{c}
k\\
q-1
\end{array}\right)(e^{-\Lambda T})^{q}(1-e^{-\Lambda T})^{k-q}\,.\label{eq:loi-de-tilde-K-T}
\end{equation}

\item \label{lemKtilde:05}Conditionally to $\widetilde{\mathcal{K}}_{T}$,
the processes $\left((\widetilde{L}_{t}^{\{2i-1,2i\}})_{0\leq t\leq T}\right)_{i\in\{1,\dots,q/2\}}$
are independent non homogeneous Poisson processes of rates, respectively,
\[
\left(\frac{\widetilde{K}_{t}^{2i-1}\widetilde{K}_{t}^{2i}}{N-1}\right)_{0\leq t\leq T}\,.
\]
 
\end{enumerate}
\end{lemma} 

Let us carry on with the definition of $(\widetilde{Z})$.

\begin{definition} \label{def:z-tilde}Let $\widetilde{k}'=\sup\left\{ k,\widetilde{T}_{k}<\infty\right\} $.
The interaction times of the $\tZ^{i}$ are $T-\widetilde{T}_{\widetilde{k}'}\leq T-\widetilde{T}_{\widetilde{k}'-1}\leq\dots\leq T-\widetilde{T}_{1}$.
\begin{itemize}
\item The $(\widetilde{Z}_{0}^{i})$ are i.i.d. $\sim\widetilde{P}_{0}$. 
\item Between the jump times, the $\widetilde{Z}^{i}$ evolve independently
of each other according to the Markov generator $L$. 
\item At a jump time $T-\widetilde{T}_{k}$, $(\tZ)$ undergoes a jump like
in Definition \ref{Def:system}, (\ref{Def:point3}), with $i,j$
replaced by $\widetilde{r}(k),\widetilde{j}(k)$. %
{} 
\end{itemize}
In so doing, we have coupled the interaction times of the systems
$(Z_{0:T}^{i})_{i\geq0}$, $(\widetilde{Z}_{0:T}^{i})_{i\geq0}$.
We can couple further and assume that for all $i$, $\widetilde{Z}_{0:T}^{i}$
and $Z_{0:T}^{i}$ coincide on the event $\{\widetilde{T}{}_{k},k\geq1\}\cap\left\{ \widetilde{T}'_{k},k\geq1\right\} =\emptyset$
(in which case, $\{T_{k},k\geq0\}=\{\widetilde{T}_{k},k\geq0\}$).

\end{definition}

\begin{definition}\label{def:z-tilde-tilde}We define the auxiliary
system $(\widetilde{\widetilde{Z}}_{0:T})_{i\geq0}$ such that
\begin{itemize}
\item it has interactions at times $\{T-\widetilde{T}_{k},k\geq1\}\setminus\{T-T''_{k},T-\widetilde{T}''_{k},k\geq1\}$ 
\item the rest of the definition is the same as for $(\widetilde{Z}_{0:T}^{i})_{i\geq0}$. 
\end{itemize}
In so doing, we have coupled the interaction times of the systems
$(\widetilde{Z}_{0:T}^{i})_{i\geq0}$, $(\widetilde{\widetilde{Z}}_{0:T}^{i})_{i\geq0}$.
We can couple further and assume that for all $i$, $\widetilde{Z}_{0:T}^{i}$
and $\widetilde{\widetilde{Z}}_{0:T}^{i}$ coincide on the event $\{\widetilde{T}{}_{k},k\geq1\}\cap\left\{ T''_{k},\widetilde{T}''_{k},k\geq1\right\} =\emptyset$.

\end{definition}

It is worth noting that the laws of $(Z_{0:T}^{1},\dots,Z_{0:T}^{q})$
and $(\widetilde{Z}_{0:T}^{1},\dots,\widetilde{Z}_{0:T}^{q})$ and
$(\widetilde{\widetilde{Z}}_{0:T}^{1},\dots,\widetilde{\widetilde{Z}}_{0:T}^{q})$
are exchangeable.

The $q$-uple $(\widetilde{Z}_{0:T}^{i})_{1\leq i\leq q}$ is obtained
from $(Z_{0:T}^{i})_{1\leq i\leq q}$ by adding links. The $q$-uple
$(\widetilde{\widetilde{Z}}_{0:T}^{i})_{1\leq i\leq q}$ is obtained
from $(\widetilde{Z}_{0:T}^{i})_{1\leq i\leq q}$ by erasing the loops.
When $N\rightarrow+\infty$, the probability that $(Z_{0:T}^{1},\dots,Z_{0:T}^{q})=(\widetilde{\widetilde{Z}}_{0:T}^{1},\dots,\widetilde{\widetilde{Z}}_{0:T}^{q})$
goes to $1$. The law of $(\widetilde{\widetilde{Z}}_{0:T}^{1},\dots,\widetilde{\widetilde{Z}}_{0:T}^{q})$
does not depend on $N$ (see its law in the theorem below). In fact,
$(\widetilde{\widetilde{Z}}^{i})_{1\leq i\leq q}$ is the asymptotic
object appearing in the limit results (see Proposition \ref{Prop:convq2},
Corollary \ref{Cor:wickproduct}). The following result can be found
in \cite{graham-meleard-1997} (Section 3.4, p. 124) (or, equivalently
\cite{graham-meleard-1994}, Section 5).

\begin{theorem}\label{rem:thgm}

The variable $\widetilde{\widetilde{Z}}_{0:T}^{1}$ has the law $\widetilde{P}_{0:T}$
(recall the definition  $\widetilde{P}$ from below Definition \ref{def:pb-02})\@.

\end{theorem}

\begin{figure}
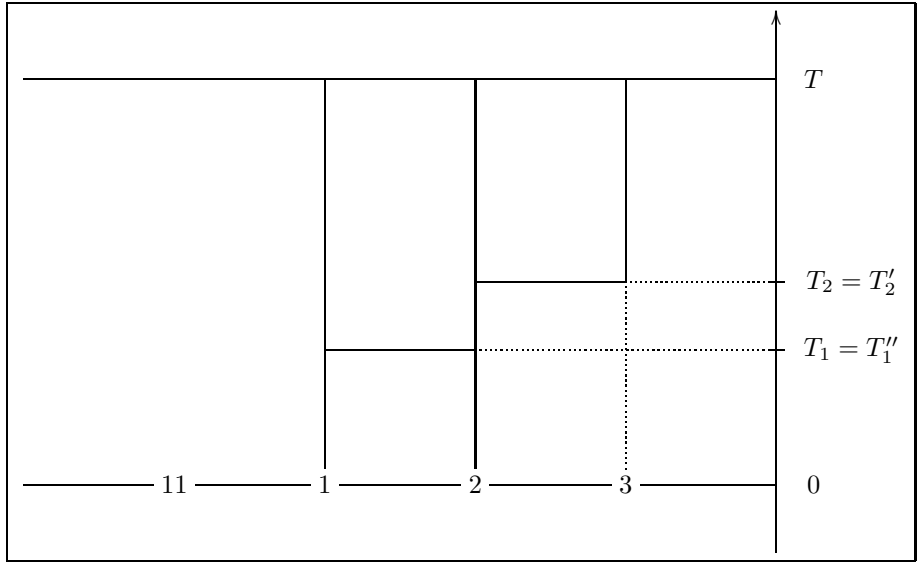

\caption{Interaction graph for $(Z_{0:T}^{N,1},Z_{0:T}^{N,2})$}

\label{fig:02}

\[\fbox{ \xygraph{!{<0cm,0cm>;<1cm,0cm>:<0cm,0.9cm>::}!{(1,0)}*{}="10"!{(2,0)}{11}="20"!{(4,0)}{1}="40"!{(6,0)}{2}="60"!{(8,0)}{3}="80"!{(10,0)}*{}="100"!{(0,6)}*{}="06"!{(2,6)}*{}="26"!{(4,6)}*{}="46"!{(6,6)}*{}="66"!{(8,6)}*{}="86"!{(4,2)}*{}="42"!{(6,2)}*{}="62"!{(6,3)}*{}="63"!{(8,3)}*{}="83"!{(10,2)}*{}="102"!{(9.9,2)}*{}="g1"!{(10.1,2)}*{}="d1"!{(11,2)}{T_{1}=T''_{1}}="T1"!{(10,3)}*{}="103"!{(10,6)}*{}="106"!{(10,7)}*{}="107"!{(9.9,3)}*{}="g2"!{(10.1,3)}*{}="d2"!{(11,3)}{T_{2}=T'_{2}}!{(10.5,0)}{0}!{(0,0)}*{}="00"!{(10.5,6)}{T}!{(10,-1)}*{}="10-1""40"-"46""60"-"66""83"-"86""80"-@{.}"83""42"-"62""63"-"83""00"-"20""20"-"40""40"-"60""60"-"80""80"-"100""06"-"106""10-1"-@{->}"107""g1"-"d1""62"-@{.}"g1""g2"-"d2""83"-@{.}"g2"} } \]
\end{figure}

\begin{figure}
\caption{Interaction graph for $(\widetilde{Z}_{0:T}^{N,1},\widetilde{Z}_{0:T}^{N,2})$}

\label{fig:03}

\[\fbox{ \xygraph{!{<0cm,0cm>;<1cm,0cm>:<0cm,0.9cm>::}!{(0,0)}*{}="00"!{(2,0)}{11}="20"!{(4,0)}{1}="40"!{(6,0)}{2}="60"!{(8,0)}{3}="80"!{(10,0)}*{}="100"!{(0,6)}*{}="06"!{(2,6)}*{}="26"!{(4,6)}*{}="46"!{(6,6)}*{}="66"!{(8,6)}*{}="86"!{(4,2)}*{}="42"!{(6,2)}*{}="62"!{(6,3)}*{}="63"!{(8,3)}*{}="83"!{(2,4)}*{}="24"!{(2,5)}*{}="25"!{(4,5)}*{}="45"!{(4,4)}*{}="44"!{(10,2)}*{}="102"!{(9.9,2)}*{}="g1"!{(10.1,2)}*{}="d1"!{(11,2)}{\widetilde{T}_{1}=T''_{1}}="T1"!{(10,3)}*{}="103"!{(10,6)}*{}="106"!{(10,7)}*{}="107"!{(9.9,3)}*{}="g2"!{(10.1,3)}*{}="d2"!{(11,3)}{\widetilde{T}_{2}=T'_{2}}!{(10.5,0)}{0}!{(0,0)}*{}="00"!{(10.5,6)}{T}!{(10,4)}*{}="104"!{(9.9,4)}*{}="g3"!{(10.1,4)}*{}="d3"!{(11,4)}{\widetilde{T}_{3}=\widetilde{T}'_{3}}!{(10,-1)}*{}="10-1"!{(2,5)}*{}="25"!{(4,5)}*{}="45"!{(9.9,5)}*{}="995"!{(10.1,5)}*{}="1015"!{(11,5)}{\widetilde{T}_4=\widetilde{T}''_4}="abo""40"-"46""60"-"66""83"-"86""80"-@{.}"83""42"-"62""63"-"83""20"-@{.}"24""24"-"26""24"-"44""00"-"20""20"-"40""40"-"60""60"-"80""80"-"100""06"-"106""10-1"-@{->}"107""g1"-"d1""62"-@{.}"g1""g2"-"d2""83"-@{.}"g2""44"-@{.}"104""g3"-"d3""25"-"45""45"-@{.}"995""995"-"1015"} }\]
\end{figure}

\begin{figure}
\caption{Interaction graph for $(\widetilde{\widetilde{Z}}_{0:T}^{N,1},\widetilde{\widetilde{Z}}_{0:T}^{N,2})$}

\label{fig:04}\[\fbox{ \xygraph{!{<0cm,0cm>;<1cm,0cm>:<0cm,0.9cm>::}!{(1,0)}*{}="10"!{(2,0)}{11}="20"!{(4,0)}{1}="40"!{(6,0)}{2}="60"!{(8,0)}{3}="80"!{(10,0)}*{}="100"!{(0,6)}*{}="06"!{(2,6)}*{}="26"!{(4,6)}*{}="46"!{(6,6)}*{}="66"!{(8,6)}*{}="86"!{(4,2)}*{}="42"!{(6,2)}*{}="62"!{(6,3)}*{}="63"!{(8,3)}*{}="83"!{(10,2)}*{}="102"!{(9.9,2)}*{}="g1"!{(10.1,2)}*{}="d1"!{(10,3)}*{}="103"!{(10,6)}*{}="106"!{(10,7)}*{}="107"!{(9.9,3)}*{}="g2"!{(10.1,3)}*{}="d2"!{(11,3)}{\widetilde{T}_{2}={T}'_{2}}!{(10.5,0)}{0}!{(0,0)}*{}="00"!{(10.5,6)}{T}!{(10,-1)}*{}="10-1"!{(2,4)}*{}="24"!{(4,4)}*{}="44"!{(9.9,4)}*{}="994"!{(10.1,4)}*{}="1014"!{(11,4)}{\widetilde{T}_3=\widetilde{T}'_3}="end""40"-"46""60"-"66""83"-"86""80"-@{.}"83""63"-"83""00"-"20""20"-"40""40"-"60""60"-"80""80"-"100""06"-"106""10-1"-@{->}"107""83"-@{.}"g2""24"-"26""20"-@{.}"24""24"-"44""44"-@{.}"994""994"-"1014"} } \]
\end{figure}

Suppose $q=2$. The figures \ref{fig:02}, \ref{fig:03}, \ref{fig:04}
are realizations of the interaction graphs for $(Z_{0:T}^{1},Z_{0:T}^{2})$,
$(\widetilde{Z}_{0:T}^{1},\widetilde{Z}_{0:T}^{2}$), $(\widetilde{\widetilde{Z}}_{0:T}^{1},\widetilde{\widetilde{Z}}_{0:T}^{2})$,
$N=10$, for the same $\omega$.

\subsection{Proof of Theorem \ref{Theo:boltzmann} (asymptotic development in
the propagation of chaos)\label{sub:Proof-of-Theorem-dl}}

\label{Sec:expansion}

\begin{proof} By Lemma \ref{Lem:equlaw1}, we have for all $l_{0}\geq0$:

\begin{eqnarray*}
\E(F(\bZ_{T}^{1},\dots,\bZ_{T}^{q})) & = & \sum_{0\leq l\leq l_{0}}[\E(F(Z_{T}^{1},\dots,Z_{T}^{q})|L_{T}=l)\p(L_{T}=l)]\\
 &  & +\E(F(Z_{T}^{1},\dots,Z_{T}^{q})|L_{T}\geq l_{0}+1)\\
 &  & \qquad\times\p(L_{T}\geq l_{0}+1).
\end{eqnarray*}
Hence, we take, for any bounded measurable $F$, 
\begin{equation}
\Delta_{q,T}^{N,l}(F)=\E(F(Z_{0:T}^{1},\dots,Z_{0:T}^{q})|L_{T}=l)\p(L_{T}=l)N^{l}\ \label{eq:def-delta}
\end{equation}
 
\begin{equation}
\overline{\Delta}_{q,T}^{N,l}(F)=\E(F(Z_{0:T}^{1},\dots,Z_{0:T}^{q})|L_{T}\geq l)\p(L_{T}\geq l)N^{l}\ .\label{eq:def-delta-barre}
\end{equation}
It is sufficient for the proof to show that $\p(L_{T}\geq l)$ is
of order $\leq1/N^{l}$, for all $l\in\N^{*}$. We write $N^{\lambda}$
for an inhomogeneous Poisson process of rate $(\lambda_{s})_{s\geq0}$.
Using Lemma \ref{Lem:law-K-tilde}, we have 
\begin{eqnarray}
\p(L_{T}\geq l) & \leq & \p(\widetilde{L}_{T}\geq l)\nonumber \\
 & = & \E(\p_{\widetilde{\mathcal{K}}_{t}}(N_{T}^{\Lambda\widetilde{K}_{.}(\widetilde{K}_{.}-1)/(N-1)}\geq l))\nonumber \\
 & \leq & \E(\p_{\widetilde{\mathcal{K}}_{t}}(N_{T}^{\Lambda\widetilde{K}_{T}(\widetilde{K}_{T}-1)/(N-1)}\geq l)\nonumber \\
 & \leq & \E\left(\frac{1}{l!}\left(\frac{\Lambda T\widetilde{K}_{T}(\widetilde{K}_{T}-1)}{N-1}\right)^{l}\right)\,.\label{eq:maj-L}
\end{eqnarray}
 And $\E((\widetilde{K}_{T})^{2l})<\infty$ by (\ref{eq:loi-de-tilde-K-T})
of Lemma \ref{Lem:law-K-tilde}.

\end{proof}

Note that we have the following bounds (for all $F\in C_{b}^{+}(\mathbb{D}([0,T],\R^{d})^{q})$)
\begin{equation}
\sup(\Delta_{q,T}^{N,l}(F),\overline{\Delta}_{q,T}^{N,l}(F))\leq\frac{\left(\Lambda T\right)^{l}}{l!}\E(\widetilde{K}_{T}^{2l})\Vert F\Vert_{\infty}e^{\frac{l}{N-1}}<\infty\,.\label{eq:borne-delta-delta-barre}
\end{equation}

\section{Rate of convergence for centered functions\label{sec:Rate-of-convergence}}

\label{Sec:wick}

\subsection{Definitions and results}

\begin{figure}
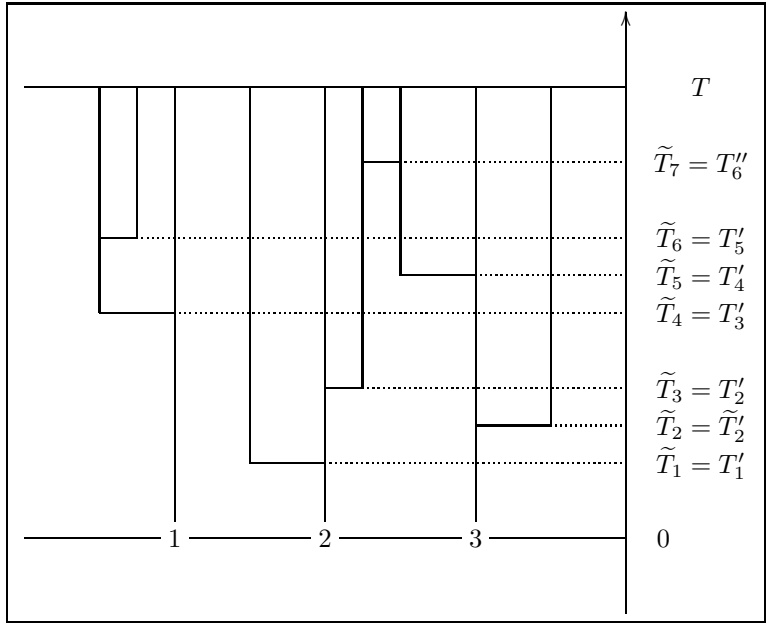

\caption{\label{fig:interaction-graph}$q=3$, interaction graph for $(\widetilde{Z}_{0:T}^{1},\widetilde{Z}_{0:T}^{2},\widetilde{Z}_{0:T}^{3})$}

\label{fig:05}

\[\fbox{ \xygraph{!{<0cm,0cm>;<1cm,0cm>:<0cm,1cm>::}!{(0,0)}*{}="00"!{(2,0)}{1}="20"!{(4,0)}{2}="40"!{(6,0)}{3}="60"!{(8,0)}*{}="80"!{(2,6)}*{}="26"!{(4,6)}*{}="46"!{(6,6)}*{}="66"!{(8,-1)}*{}="8-1"!{(8,7)}*{}="87"!{(0,6)}*{}="06"!{(8,6)}*{}="86"!{(1,3)}*{}="13"!{(8,3)}*{}="83"!{(1,6)}*{}="16"!{(1.5,4)}*{}="154"!{(1,4)}*{}="14"!{(1.5,6)}*{}="156"!{(8,4)}*{}="84"!{(2,3)}*{}="23"!{(1,4)}="14"!{(3,1)}*{}="31"!{(4,1)}*{}="41"!{(8,1)}*{}="81"!{(3,6)}*{}="36"!{(4.5,2)}*{}="452"!{(8,2)}*{}="82"!{(4.5,6)}*{}="456"!{(8,4.5)}*{}="845"!{(4,2)}*{}="42"!{(6,3.5)}*{}="635"!{(8,3.5)}*{}="835"!{(5,3.5)}*{}="535"!{(8,3.5)}*{}="835"!{(5,6)}*{}="56"!{(5,5)}*{}="55"!{(4.5,5)}*{}="455"!{(8,5)}*{}="85"!{(6,1.5)}*{}="615"!{(7,1.5)}*{}="715"!{(7,6)}*{}="76"!{(8,1.5)}*{}="815"!{(8.5,0)}{0}!{(9,1)}{\widetilde{T}_{1}=T'_{1}}!{(9,1.5)}{\widetilde{T}_{2}=\widetilde{T}'_{2}}!{(9,2)}{\widetilde{T}_{3}=T'_{2}}!{(9,3)}{\widetilde{T}_{4}=T'_{3}}!{(9,3.5)}{\widetilde{T}_{5}=T'_{4}}!{(9,4)}{\widetilde{T}_{6}=T'_{5}}!{(9,5)}{\widetilde{T}_{7}=T''_{6}}!{(9,6)}{T}"00"-"20""20"-"40""40"-"60""60"-"80""06"-"86""8-1"-@{->}"87""13"-"16""13"-"23""20"-"26""154"-"156""154"-"14""40"-"46""31"-"36""31"-"41""452"-"456""452"-"42""60"-"66""635"-"535""535"-"56""55"-"455""615"-"715""715"-"76""23"-@{.}"83""154"-@{.}"84""41"-@{.}"81""452"-@{.}"82""55"-@{.}"85""635"-@{.}"835""715"-@{.}"815"} }\]
\end{figure}

For $i\in[q]$, $I\subset[q]$ such that $i\in I$, we define the
event
\[
A_{i}^{I}=\left\{ \left\{ \widetilde{T}_{k}:k\geq0,\widetilde{T}_{k}\leq T,\widetilde{T}_{k}\in\mathcal{T}^{I},\widetilde{r}(k)\mbox{ or }\widetilde{j}(k)\in\widetilde{C}_{\widetilde{T}_{k}-}^{i}\right\} \subset\left\{ T'_{k},k\geq1\right\} \right\} \,,
\]
and we set
\[
A_{i}=A_{i}^{[q]}\,.
\]
Recall the definition of $\mathcal{T}^{I}$ from  (\ref{eq:def-Tronde}).
For all $i,l$ such that $1\leq i<q$, $1\leq l\leq q$, we define
\[
\widetilde{L}_{i,i+i}=\left\{ \#\left(\left\{ \widetilde{T}_{k}:\widetilde{T}_{k}\leq T,\#\left(\{\widetilde{r}(k),\widetilde{j}(k)\}\cap(\widetilde{C}_{\widetilde{T}_{k}-}^{i}\cup\widetilde{C}_{\widetilde{T}_{k}-}^{i+1})\right)=2,k\geq0\right\} \right)=1\right\} \,,
\]

\[
\mbox{for }q\mbox{ even, }\widetilde{L}_{1,q}=\widetilde{L}_{1,2}\cap\dots\cap\widetilde{L}_{q-1,q}\cap\left\{ \#\left(\left\{ \widetilde{T}_{k},k\geq1\right\} \cap\left\{ T''_{l},\widetilde{T}''_{l},l\geq1\right\} \right)=q/2\right\} \,,
\]

\[
E_{T}^{l}=\#\left(\{\widetilde{T}_{k},k\geq1,\widetilde{T}_{k}\leq T:\widetilde{r}(k)\in\widetilde{C}_{\widetilde{T}_{k}-}^{l}\}\cap\{\widetilde{T}'_{i},i\geq1\}\right)\,.
\]
Note that we do not write anything about $\widetilde{j}$ in the last
definition because $\widetilde{r}$ and $\widetilde{j}$ do not play
symmetrical roles (see Section \ref{sub:Auxiliary-systems}). For
a fixed $\omega$, we define an relation on $\N$ by 
\[
i\bowtie j\mbox{ if }\exists r\mbox{ such that }\widetilde{T}_{r}\in\{\widetilde{T}''_{l},T''_{l},l\geq1\}\mbox{ and }\#\left(\{\widetilde{r}(r),\widetilde{j}(r)\}\cap\left(\widetilde{C}_{\widetilde{T}_{r}-}^{i}\cap\widetilde{C}_{\widetilde{T}_{r}-}^{j}\right)\right)=2\,.
\]
We extend it into an equivalence relation by imposing $i\bowtie i$,
($i\bowtie j$ and $j\bowtie k$) $\Longrightarrow$ $i\bowtie k$.
For all $i\in[k]$, $\mathcal{C}_{i}$ denote the class of $i$ for
the relation $\bowtie$. For $I\subset[q]$, $k\in[q]$, we define
\begin{equation}
L_{I}=\{\mathcal{C}_{\max(I)}\cap[q]=I\}\,,\label{eq:def-L-I}
\end{equation}

\begin{multline}
L_{I}^{k}=L_{I}\cap(A_{\max(I)}^{I})^{c}\cap\left(\cap_{1\leq i\leq k,i\in I}(A_{i}^{[k]})^{c}\right)\cap\left(\cap_{k+1\leq i\leq q,i\in I\backslash\max(I)}A_{i}^{([k]\cup I)\backslash\max(I)}\right)\,.\label{eq:def-L-I-k-2}
\end{multline}

Let us have a look at Figure \ref{fig:05} to clarify the notions
above. Suppose $\omega\in\Omega$ is such that the graph in Figure
\ref{fig:interaction-graph} occurs. Note that: $\omega\in A_{1}$,
$\omega\in A_{2}^{c}$, $\omega\in A_{3}^{c}$, $\omega\in A_{2}^{\{1,2\}}$,
$\mathcal{C}_{1}=\{1\}$, $\mathcal{C}_{2}=\{2,3\}$, $E_{T}^{3}=1$.

Let us now write some sentences designed to illustrate the meaning
of the definitions above. Let $I\subset[q]$. When a time of $\mathcal{T}^{I}$
is a jump time for $\widetilde{C}^{\widetilde{r}},\widetilde{C}^{\widetilde{j}}$,
we say that a link between $\widetilde{C}^{\widetilde{r}}$ and $\widetilde{C}^{\widetilde{j}}$
is formed at this time%
; and if $\widetilde{r},\widetilde{j}\in I$, we say that this is
an internal link (or a loop) of $(\widetilde{C}^{i})_{i\in I}$. We
can define the same notions for the processes $(C^{k})_{1\leq k\leq q}$.
Using this terminology, the event $A_{l}^{I}$ (with $l\in I$) is
the event that the links happening at a time in $\mathcal{T}^{I}$
and between $\widetilde{C}^{l}$ and the rest are external links of
$\widetilde{C}^{l}$ and do not belong to $\{\widetilde{T}'_{k},k\geq1\}$.
The event $\widetilde{L}_{i,i+1}$ is the event that there is exactly
one link between $\widetilde{C}^{i}$ and $\widetilde{C}^{i+1}$.
The event $\widetilde{L}_{1,q}$ is the event that there is exactly
one link between $\widetilde{C}^{2i-1}$ and $\widetilde{C}^{2i}$
for all $i\in[q/2]$, and that $\widetilde{L}_{T}$%
{} is $q/2$. The relation $i\bowtie j$ expresses that $\widetilde{C}^{i}$
and $\widetilde{C}^{j}$ are linked by a string of links.

\subsection{\label{sub:Technical-lemmas} Technical lemmas}

Before going into the proof of Proposition \ref{Prop:convq2}, we
need some technical results.

\begin{lemma}\label{lem:maj-E-K-tilde}For all $l\in\mathbb{N}^{*}$,
\[
\E(\widetilde{K}_{T}^{l})\leq\frac{(q+l-1)!}{(e^{-\Lambda T})^{l}(1-e^{-\Lambda T})(q-1)!}\,.
\]

\end{lemma}

\begin{proof}

By (\ref{eq:loi-de-tilde-K-T}), with $\alpha=e^{-\Lambda T}$:
\begin{eqnarray*}
\E(\widetilde{K}_{T}^{l}) & = & \sum_{k=q}^{+\infty}k^{l}\left(\begin{array}{c}
k\\
q-1
\end{array}\right)\alpha^{q}(1-\alpha)^{k-q}\\
 & \leq & \frac{\alpha^{q}}{(1-\alpha)(q-1)!}\sum_{k=q}^{+\infty}(k+l)(k+l-1)\dots(k-q+2)(1-\alpha)^{k-q+1}\\
 & \leq & \frac{\alpha^{q}}{(1-\alpha)(q-1)!}\frac{(q+l-1)!}{\alpha^{q+l}}\,.
\end{eqnarray*}

\end{proof}

\begin{lemma}\label{lem:ens-negligeable}For all $r\in\{0,1,\dots,q\}$,
$i_{1},\dots,i_{r}\in[q]$, $l\geq1$,
\[
\p(E_{T}^{i_{1}}\geq1,\dots,E_{T}^{i_{r}}\geq1,\widetilde{L}_{T}\geq l)\leq\frac{(\Lambda T)^{r+l}(q+2r+2l-1)!}{r^{r}(N-1)^{r+l}(e^{-\Lambda T})^{2r+2l}(1-e^{-\Lambda T})(q-1)!l!}\,.
\]

\end{lemma}

\begin{proof}

When $\mathcal{K}_{T}$  is fixed, for any $i$, the law of $\1_{E_{T}^{i}\geq1}$
is Bernoulli of parameter 
\[
1-\exp\left(-\int_{0}^{T}\Lambda K_{t}^{i}-\frac{\Lambda K_{t}^{i}(N-K_{t})_{+}}{N-1}dt\right)
\]
 and $\1_{E_{T}^{i_{1}}\geq1},\dots,\1_{E_{T}^{i_{r}}\geq1}$ are
independent. So we have for all $t$, $\omega$, 
\begin{eqnarray*}
\E_{\mathcal{K}_{T}}\left(\prod_{j=1}^{r}\1_{E_{T}^{i_{j}}\geq1}\right) & \leq & \prod_{j=1}^{r}\frac{\Lambda TK_{T}^{i_{j}}K_{T}}{N-1}\\
\mbox{(as }\sum_{j=1}^{r}K_{T}^{i_{j}}\leq\widetilde{K}_{T}\mbox{)} & \leq & \frac{(\Lambda T)^{r}(\widetilde{K}_{T})^{2r}}{r^{r}(N-1)^{r}}\,.
\end{eqnarray*}
Basic estimates then leads to $ $$ $$\E_{\widetilde{\mathcal{K}}_{T}}\left(\prod_{j=1}^{r}\1_{E_{T}^{i_{j}}\geq1}\right)\leq\frac{(\Lambda T)^{r}(\widetilde{K}_{T})^{2r}}{r^{r}(N-1)^{r}}$.
 We have, as in (\ref{eq:maj-L}), 
\[
\p(\widetilde{L}_{T}\geq l|\widetilde{\mathcal{K}}_{T})\leq\frac{1}{l!}\left(\frac{\Lambda T\widetilde{K}{}_{T}(\widetilde{K}{}_{T}-1)}{N-1}\right)^{l}\,.
\]
Now, as $E_{T}^{i_{1}},\dots,E_{T}^{i_{r}},\widetilde{L}_{T}$ are
independent conditionally to $\widetilde{\mathcal{K}}_{T}$, by Lemma
\ref{lem:maj-E-K-tilde}, we obtain: 
\begin{multline*}
\p(E_{T}^{i_{1}}\geq1,\dots,E_{T}^{i_{r}}\geq1,\widetilde{L}_{T}\geq l)\leq\E\left(\frac{(\Lambda T)^{r+l}(\widetilde{K}_{T})^{2r+2l}}{r^{r}(N-1)^{r+l}l!}\right)\\
\leq\frac{(\Lambda T)^{r+l}(q+2r+2l-1)!}{r^{r}(N-1)^{r+l}l!(e^{-\Lambda T})^{2r+2l}(1-e^{-\Lambda T})(q-1)!}\,.
\end{multline*}

\end{proof}

We write $\mathcal{P}_{m}$ for the set of $m$-uples $(I_{1},\dots,I_{m})$
of subsets of $[q]$ partitioning $[q]$ and such that  $\max(I_{1})>\max(I_{2})>\dots>\max(I_{m})$.\begin{definition}

We define the auxiliary particle systems $(Z_{0:T}^{k,1},Z_{0:T}^{k,2},\dots,Z_{0:T}^{k,k})$
(one system for each $k\in[q]$) by saying that
\begin{itemize}
\item it has interactions at times $\left\{ T-T_{k},k\geq1\right\} \cap\left\{ T-t,t\in\mathcal{T}^{\{1,k\}}\right\} $, 
\item the rest of the definition is the same as for $(\widetilde{Z}_{0:T}^{i})$. 
\end{itemize}
In so doing, we have coupled the interaction times of $(Z_{0:T}^{k,i})_{1\leq i\leq k}$
with the interaction times of the other systems. We can couple further
and assume that $(Z_{0:T}^{1},\dots,Z_{0:T}^{k})$ and $(Z_{0:T}^{k,1},\dots,Z_{0:T}^{k,k})$
coincide on the event $\left\{ T-T_{k},k\geq1\right\} \cap\left\{ T-t,t\in\mathcal{T}^{\{1,k\}}\right\} =\left\{ T-T_{k},k\geq1\right\} $.
Note that for all $k$, the law of $(Z_{0:T}^{k,1},\dots,Z_{0:T}^{k,k})$
is exchangeable.

\end{definition}

\begin{lemma}\label{Lem:qlim} If $q$ is odd then for all $k\in\{0,\dots,q\}$,
for all $F\in\mathcal{B}_{0}^{sym}(q)$, $m\in[q]$, $(I_{1},\dots,I_{m})\in\mathcal{P}_{m}$,
\[
N^{q/2}\E(F(Z_{0:T}^{k,1},\dots,Z_{0:T}^{k,k},\widetilde{\widetilde{Z}}_{0:T}^{k+1},\dots,\widetilde{\widetilde{Z}}_{0:T}^{q})\prod_{i\in[m]}\1_{L_{I_{i}}^{k}})\convN0\,,
\]
(we use the following convention: in the case $k=0$,\\
 $\E(F(Z_{0:T}^{k,1},\dots,Z_{0:T}^{k,k},\widetilde{\widetilde{Z}}_{0:T}^{k+1},\dots,\widetilde{\widetilde{Z}}_{0:T}^{q}))=\E(\widetilde{\widetilde{Z}}_{0:T}^{1},\dots,\widetilde{\widetilde{Z}}_{0:T}^{q})$).

If $q$ is an even integer and $k\in[q]$,

\begin{multline}
N^{q/2}\E(F(Z_{0:T}^{k,1},\dots,Z_{0:T}^{k,k},\widetilde{\widetilde{Z}}_{0:T}^{k+1},\dots,\widetilde{\widetilde{Z}}_{0:T}^{q})\prod_{i\in[m]}\1_{L_{I_{i}}^{k}})\\
\convN\E\left(\E_{\widetilde{\mathcal{K}}_{T}}\left(F(\widetilde{Z}_{0:T}^{1},\dots,\widetilde{Z}_{0:T}^{k},\widetilde{\widetilde{Z}}_{0:T}^{k+1},\dots,\widetilde{\widetilde{Z}}_{0:T}^{q})|\widetilde{L}_{1,q}\right)\right.\\
\left.\times\prod_{i=1}^{q/2}\int_{0}^{T}\Lambda\widetilde{K}_{s}^{2i-1}\widetilde{K}_{s}^{2i}ds\right)\,\label{eq:conv-lacets}
\end{multline}
if 
\begin{equation}
\forall j\,,\,\#I_{j}=2\,\mbox{and}\,(I_{j}\subset[k]\mbox{ or }I_{j}\subset\{k+1,\dots,q\})\,,\label{eq:cond-sur-les-I}
\end{equation}
and 
\[
N^{q/2}\E(F(Z_{0:T}^{k,1},\dots,Z,_{0:T}^{k,k},\widetilde{\widetilde{Z}}_{0:T}^{k+1},\dots,\widetilde{\widetilde{Z}}_{0:T}^{q})\prod_{i\in[m]}\1_{L_{I_{i}}^{k}})\convN0
\]
otherwise.

Moreover, these limits still hold if we replace $(Z_{0:T}^{k,1},\dots,Z_{0:T}^{k,k})$
by $(\widetilde{\widetilde{Z}}_{0:T}^{1},\dots,\widetilde{\widetilde{Z}}_{0:T}^{k})$
and \\
$(\widetilde{Z}_{0:T}^{1},\dots,\widetilde{Z}_{0:T}^{k})$ by $(\widetilde{\widetilde{Z}}_{0:T}^{1},\dots,\widetilde{\widetilde{Z}}_{0:T}^{k})$
in the formulas above.

\end{lemma}

Note hat the expectation in (\ref{eq:conv-lacets}) does not depend
on $N$ and that it cannot be simplified by the use of the tower formula
because the conditional expectation is conditional with respect to
$\widetilde{\mathcal{K}}_{T},\widetilde{L}_{1,q}$.

\begin{proof}We define the event
\begin{equation}
B=\bigcup_{1\leq r\leq q}\bigcup_{1\leq i_{1}<\dots<i_{r}\leq q}\left(\left(\cap_{j=1}^{r}\{E_{T}^{i_{j}}\geq1\}\right)\cap\{\widetilde{L}_{T}\geq\left\lceil \frac{q-r}{2}\right\rceil \right)\,.\label{eq:def-B}
\end{equation}
Suppose that $q$ is odd or that $q$ is even and (\ref{eq:cond-sur-les-I})
does not hold. Then for all $m\in[q]$, $(I_{1},\dots,I_{m})\in\mathcal{P}_{m}$,
$k\in[q]$, 
\begin{equation}
\bigcap_{1\leq j\leq m}L_{I_{j}}^{k}\subset\left(B\cup\{\widetilde{L}_{T}\geq\frac{q}{2}+1\}\right)\,.\label{eq:inclusion-chantme}
\end{equation}
Let us write a short justification of the last formula. Suppose that
$m\in[q]$, $(I_{1},\dots,I_{m})\in\mathcal{P}_{m}$, $\omega\in\bigcap_{1\leq j\leq m}L_{I_{j}}^{k}$.
We note that, for all $I\subset[q]$, $\omega\in L_{I}\Rightarrow\widetilde{L}_{T}^{I}(\omega)\geq\#I-1$
(recall the definition of $\widetilde{L}_{T}^{I}$ from (\ref{eq:def-L-tilde-(j-1,j)})).
We note also that $\left(\omega\in L_{I}^{k}\mbox{ with }I=\{i\}\right)\Rightarrow E_{T}^{i}(\omega)\geq1$.
Indeed, for such an $\omega$, $\omega\in\left(A_{i}^{\{i\}}\right)^{c}$
so there exists a $\widetilde{T}_{r}\in\mathcal{T}^{\{i\}}$ such
that $\widetilde{T}_{r}\notin\{T'_{l},l\geq1\}$; as $\omega\in L_{\{i\}}$,
$i$ is not linked to any $i'$ in $[q]$ by the relation $\bowtie$,
so $E_{T}^{i}(\omega)\geq1$. %
{} So, we have (\ref{eq:inclusion-chantme}) for $q$ odd. If $q$ is
even and we do not have (\ref{eq:cond-sur-les-I}), then (by the same
reasoning as above):
\begin{itemize}
\item If there exists $j$ such that $\#I_{j}\neq2$, then $\omega\in B\cup\{\widetilde{L}_{T}\geq\frac{q}{2}+1\}$.
\item If $\#I_{j}=2$, for all $j$, and, say, for some $l$, $[k]\cap I_{l}\neq\emptyset$,
$\{k+1,\dots,q\}\cap I_{l}\neq\emptyset$, then $\sum_{j=1}^{m}\widetilde{L}_{T}^{I_{j}}\geq q/2$.
Let $i_{0}=\min(I_{l})$. As $\omega\in L_{I_{l}}^{k}$, we have $\omega\in(A_{i_{0}}^{[k]})^{c}$,
so: either $E_{T}^{i_{0}}\geq1$, either there exists $r\neq i_{0}$,
$r\in[k]$  such that $\widetilde{L}_{T}^{\{r,i_{0}\}}\geq1$. So,
$\omega\in B\cup\{\widetilde{L}_{T}\geq\frac{q}{2}+1\}$.
\end{itemize}
Using Lemma \ref{lem:ens-negligeable} and Equation (\ref{eq:inclusion-chantme}),
we then have
\[
N^{q/2}\p(\bigcap_{1\leq j\leq m}L_{I_{j}}^{k})\convN0\,.
\]
Now, suppose that $q$ is even and that (\ref{eq:cond-sur-les-I})
holds. Note that $(Z_{0:T}^{k,1},\dots,Z_{0:T}^{k,k})$ and $(\widetilde{Z}_{0:T}^{1},\dots,\widetilde{Z}_{0:T}^{k})$
coincide on the event $\widetilde{L}_{1,q}\cap\{E_{T}^{1}=\dots=E_{T}^{k}=0\}$.
Then, using the symmetries of the problem and Lemma \ref{lem:ens-negligeable},
we see that the limit we are looking for is the same as
\begin{multline*}
\lim_{N\rightarrow+\infty}N^{q/2}\E(F(Z_{0:T}^{k,1},\dots,Z_{0:T}^{k,k},\widetilde{\widetilde{Z}}_{0:T}^{k+1},\dots,\widetilde{\widetilde{Z}}_{0:T}^{q})\1_{\widetilde{L}_{1,q}})\\
=\lim_{N\rightarrow+\infty}N^{q/2}\E(F(\widetilde{Z}_{0:T}^{1},\dots,\widetilde{Z}_{0:T}^{k},\widetilde{\widetilde{Z}}_{0:T}^{k+1},\dots,\widetilde{\widetilde{Z}}_{0:T}^{q})\1_{\widetilde{L}_{1,q}})\,.
\end{multline*}
We set for all $j,t$ 
\[
\widetilde{L}'_{1,q}=\left\{ \widetilde{L}_{T}^{\{1,2\}}\geq1\right\} \cap\dots\cap\left\{ \widetilde{L}_{T}^{\{q-1,q\}}\geq1\right\} \,.
\]
 We set for all $j\in[q/2]$
\[
\alpha(2j-1,2j)=\exp\left(-\int_{0}^{T}\frac{\Lambda\widetilde{K}_{s}^{2j-1}\widetilde{K}_{s}^{2j}}{N-1}ds\right)\,.
\]
 We have

\begin{eqnarray}
N^{q/2}\p_{\widetilde{\mathcal{K}}_{T}}(\widetilde{L}_{1,q}') & = & \prod_{1\leq j\leq q/2}\left[N(1-\alpha(2j-1,2j))\right]\label{eq:prod-alpha}\\
 & \overset{\mbox{a.s.}}{\convN} & \prod_{1\leq j\leq q/2}\int_{0}^{T}\Lambda\widetilde{K}_{s}^{2j-1}\widetilde{K}_{s}^{2j}ds\ .\nonumber 
\end{eqnarray}
 We have 
\[
\p_{\widetilde{\mathcal{K}}_{T}}(\widetilde{L}'_{1,q}\backslash\widetilde{L}_{1,q})\leq\p_{\widetilde{\mathcal{K}}_{T}}(\widetilde{L}_{T}>q/2)\,.
\]
 So, by Lemma \ref{lem:ens-negligeable}, $N^{q/2}\p_{\widetilde{\mathcal{K}}_{T}}(\widetilde{L}'_{1,q}\backslash\widetilde{L}_{1,q})\overset{\mbox{a.s.}}{\convN}0$.
And so :
\[
N^{q/2}\p_{\widetilde{\mathcal{K}}_{T}}(\widetilde{L}_{1,q})\overset{\mbox{a.s.}}{\convN}\prod_{1\leq i\leq q/2}\int_{0}^{T}\Lambda\widetilde{K}_{s}^{2i-1}\widetilde{K}_{s}^{2i}ds\ \,.
\]
 Now,
\begin{eqnarray*}
N^{q/2}\p_{\tilde{\mathcal{K}}_{T}}(\widetilde{L}_{1,q}) & \leq & N^{q/2}\p_{\tilde{\mathcal{K}}_{T}}(\widetilde{L}_{1,q}')\\
\mbox{(by (\ref{eq:prod-alpha}))} & \leq & \left(\frac{N}{N-1}\right)^{q/2}\prod_{1\leq j\leq q/2}\int_{0}^{T}\Lambda\tilde{K}_{s}^{2i-1}\tilde{K}_{s}^{2i}ds\\
 & \leq & 2^{q/2}T^{q/2}\Lambda^{q/2}(\widetilde{K}_{T})^{q}\,,
\end{eqnarray*}
 which is of finite expectation by (\ref{eq:loi-de-tilde-K-T}). Thus,
using the dominated convergence theorem, we obtain
\begin{multline*}
\lim_{N\rightarrow+\infty}\E(\E_{\widetilde{\mathcal{K}}_{T}}(F(\widetilde{Z}_{0:T}^{1},\dots,\widetilde{Z}_{0:T}^{k},\widetilde{\widetilde{Z}}_{0:T}^{k+1},\dots,\widetilde{\widetilde{Z}}_{0:T}^{q})|\widetilde{L}_{1,q})N^{q/2}\p_{\widetilde{\mathcal{K}}_{T}}(\widetilde{L}_{1,q}))\\
=\E\left(\E_{\widetilde{\mathcal{K}}_{T}}(F(\widetilde{Z}_{0:T}^{1},\dots,\widetilde{Z}_{0:T}^{k},\widetilde{\widetilde{Z}}_{0:T}^{k+1},\dots,\widetilde{\widetilde{Z}}_{0:T}^{q})|\widetilde{L}_{1,q})\prod_{i=1}^{q/2}\int_{0}^{T}\Lambda\widetilde{K}_{s}^{2i-1}\widetilde{K}_{s}^{2i}ds\right)\,.
\end{multline*}
 By Lemma \ref{Lem:fondamental}, \ref{lem-point-01} and Lemma \ref{Lem:law-K-tilde},
\ref{lemKtilde:05}, the law of 
\[
\E_{\widetilde{\mathcal{K}}_{T}}(F(\widetilde{Z}_{0:T}^{1},\dots,\widetilde{Z}_{0:T}^{k},\widetilde{\widetilde{Z}}_{0:T}^{k+1},\dots,\widetilde{\widetilde{Z}}_{0:T}^{q})|\widetilde{L}_{1,q})
\]
 does not depend on $N$.

We do not write the proof of the last point of the lemma because it
is very similar to what is written above.

\end{proof}

\subsection{\label{sub:Proof-of-Proposition-termes-centres}Proof of Proposition
\ref{Prop:convq2}}

\begin{proof}

Since $Z_{0:T}^{i}=\widetilde{\widetilde{Z}}_{0:T}^{i}$ on $A_{i}$
(for all $i$), we have

\begin{multline}
\E(F(Z_{0:T}^{1},\dots,Z_{0:T}^{q}))=\\
\E(F(Z_{0:T}^{1},\dots,Z_{0:T}^{q})\1_{(A_{1}\cap\dots\cap A_{q})^{c}})+\E(F(Z_{0:T}^{1},\dots,Z_{0:T}^{q})\1_{A_{1}\cap\dots\cap A_{q}})=\\
\E(F(Z_{0:T}^{1},\dots,Z_{0:T}^{q})\1_{(A_{1}\cap\dots\cap A_{q})^{c}})+\E(F(\widetilde{\widetilde{Z}}_{0:T}^{1},\dots,\widetilde{\widetilde{Z}}_{0:T}^{q}))-\E(F(\widetilde{\widetilde{Z}}{}_{0:T}^{1},\dots,\widetilde{\widetilde{Z}}_{0:T}^{q})\1_{(A_{1}\cap\dots\cap A_{q})^{c}})=\\
\E(F(Z_{0:T}^{1},\dots,Z_{0:T}^{q})\1_{(A_{1}\cap\dots\cap A_{q})^{c}})-\E(F(\widetilde{\widetilde{Z}}_{0:T}^{1},\dots,\widetilde{\widetilde{Z}}_{0:T}^{q})\1_{(A_{1}\cap\dots\cap A_{q})^{c}})\,,\label{eq:dec-01}
\end{multline}
where, to obtain $\E(F(\widetilde{\widetilde{Z}}_{0:T}^{1},\dots,\widetilde{\widetilde{Z}}_{0:T}^{q}))=0$,
we have used the fact that $\widetilde{\widetilde{Z}}_{0:T}^{q}$
is independent of $\widetilde{\widetilde{Z}}_{0:T}^{1}\dots,\widetilde{\widetilde{Z}}_{0:T}^{q-1}$,
and that $\int_{\mathbb{D}([0,T],\R^{d})}F(z_{1},\dots z_{q})\widetilde{P}_{0:T}(dz_{q})=0$,
for all $z_{1},\dots,z_{q-1}$. This kind of reasoning will be used
again in the following. We have (using the symmetry of the problem)

\begin{multline*}
\E(F(Z_{0:T}^{1},\dots,Z_{0:T}^{q})\1_{(A_{1}\cap\dots\cap A_{q})^{c}})\\
=\sum_{k=1}^{q}\left(\begin{array}{c}
q\\
k
\end{array}\right)\E(F(Z_{0:T}^{1},\dots,Z_{0:T}^{q})\1_{A_{1}^{c}}\dots\1_{A_{k}^{c}}\1_{A_{k+1}}\dots\1_{A_{q}})\\
=\sum_{k=1}^{q}\left(\begin{array}{c}
q\\
k
\end{array}\right)\E(F(Z_{0:T}^{k,1},\dots,Z_{0:T}^{k,k},\widetilde{\widetilde{Z}}_{0:T}^{k+1},\dots,\widetilde{\widetilde{Z}}_{0:T}^{q})\1_{(A_{1}^{[k]})^{c}}\dots\1_{(A_{k}^{[k]})^{c}}\1_{A_{k+1}}\dots\1_{A_{q}})\,.
\end{multline*}
For all $H_{1}\subset H_{2}\subset H_{3}\subset[q]$, such that $H_{3}\backslash H_{2}\neq\emptyset$
(the equality being otherwise obvious), we have 
\begin{equation}
\cap_{i\in H_{3}\backslash H_{1}}A_{i}^{H_{3}}=\left(\cap_{i\in H_{2}\backslash H_{1}}A_{i}^{H_{2}}\right)\cap\left(\cap_{i\in H_{3}\backslash H_{2}}A_{i}^{H_{3}}\right)\,.\label{eq:ensembles-01}
\end{equation}
This will be used many times in this proof. We obtain already (with
$H_{1}=[k]$, $H_{2}=[q-1]$, $H_{3}=[q]$ ) 
\begin{multline}
\E(F(Z_{0:T}^{1},\dots,Z_{0:T}^{q})\1_{(A_{1}\cap\dots\cap A_{q})^{c}})\\
=\E(F(Z_{0:T}^{1},\dots,Z_{0:T}^{q})\1_{A_{1}^{c}\cap\dots\cap A_{q}^{c}}+\sum_{k=1}^{q-1}\left(\begin{array}{c}
q\\
k
\end{array}\right)\E(F(Z_{0:T}^{k,1},\dots,Z_{0:T}^{k,k},\widetilde{\widetilde{Z}}_{0:T}^{k+1},\dots,\widetilde{\widetilde{Z}}_{0:T}^{q})\\
\times\prod_{i=1}^{k}\1_{(A_{i}^{[k]})^{c}}\prod_{i=k+1}^{q-1}\1_{A_{i}^{[q-1]}}\times(1-\1_{A_{q}^{c}}))\,.\label{eq:ensembles-05}
\end{multline}
We look at one term in the last sum above for a fixed $k\leq q-1$.
As $\widetilde{\widetilde{Z}}_{0:T}^{q}$ is independent of $Z_{0:T}^{k,1},\dots,Z_{0:T}^{k,k},\widetilde{\widetilde{Z}}_{0:T}^{k+1},\dots,\widetilde{\widetilde{Z}}_{0:T}^{q-1}$
and $\1{}_{A_{1}^{[k]}}$, \ldots,$\1_{A_{k}^{[k]}}$,$\1_{A_{k+1}^{[q-1]}}$,
\ldots, $\1_{A_{q-1}^{[q-1]}}$, we obtain that this quantity is
equal to 
\begin{equation}
\E(F(Z_{0:T}^{k,1},\dots,Z_{0:T}^{k,k},\widetilde{\widetilde{Z}}_{0:T}^{k+1},\dots,\widetilde{\widetilde{Z}}_{0:T}^{q})\1_{(A_{1}^{[k]})^{c}}\dots\1_{(A_{k}^{[k]})^{c}}\times\1_{A_{k+1}^{[q-1]}}\dots\1_{A_{q-1}^{[q-1]}}(-\1_{A{}_{q}^{c}}))\,.\label{eq:ensembles-05bis}
\end{equation}
Let us set $i_{1}=q$. Let $I_{1}\subset[q]$. Recall the definition
of $L_{I_{1}}^{k}$ from (\ref{eq:def-L-I-k-2}) and the definition
of $L_{I_{1}}$ from (\ref{eq:def-L-I}). For all $I_{1}\subset[q]$
such that $q\in I_{1}$, we have
\begin{equation}
A_{q}^{c}\cap L_{I_{1}}=(A_{q}^{I_{1}})^{c}\cap L_{I_{1}}\,,\label{eq:ins-I}
\end{equation}
then (using (\ref{eq:ensembles-01}) with $H_{1}=[k]$, $H_{2}=[k]\cup(I_{1}\cap\{k+1,\dots,q-1\})$,
$H_{3}=[q-1]$)
\begin{equation}
\cap_{k+1\leq i\leq q-1}A_{i}^{[q-1]}=\left(\cap_{i\in I_{1}\cap\{k+1,\dots,q-1\}}A_{i}^{[k]\cup(I_{1}\cap\{k+1,\dots,q-1\})}\right)\cap\left(\cap_{k+1\leq i\leq q-1,i\notin I_{1}}A_{i}^{[q-1]}\right)\,.\label{eq:use-H}
\end{equation}
We have:
\begin{multline}
\left(\cap_{1\leq i\leq k}(A_{i}^{[k]})^{c}\right)\cap\left(\cap_{k+1\leq i\leq q-1}A_{i}^{[q-1]}\right)\cap A_{q}^{c}=\\
\bigsqcup_{I_{1}\subset[q]\,,\,q\in I_{1}}\left[\left(\cap_{1\leq i\leq k}(A_{i}^{[k]})^{c}\right)\cap\left(\cap_{k+1\leq i\leq q-1}A_{i}^{[q-1]}\right)\cap A_{q}^{c}\cap L_{I_{1}}\right]\,.\label{eq:ensembles-02}
\end{multline}
(the symbol $\sqcup$ means ``disjoint union''). For each term in
the union above, we have (using (\ref{eq:ins-I}), (\ref{eq:use-H}))
\begin{multline}
\left(\cap_{1\leq i\leq k}(A_{i}^{[k]})^{c}\right)\cap\left(\cap_{k+1\leq i\leq q-1}A_{i}^{[q-1]}\right)\cap A_{q}^{c}\cap L_{I_{1}}\\
=\left(\cap_{i\in[k],i\notin I_{1}}(A_{i}^{[k]})^{c}\right)\cap\left(\cap_{k+1\leq i\leq q-1,i\notin I_{1}}A_{i}^{[q-1]}\right)\cap L_{I_{1}}^{k}\,.\label{eq:ensembles-03}
\end{multline}
\[
\]
 For each $I_{1}\subset[q]$ such that $q\in I_{1}$, we set $i_{2}=\max([q]\backslash I_{1})$,
if it exists. 

If $i_{2}\geq k+1$, we can then write (using (\ref{eq:ensembles-01})
with $H_{3}=[q-1]$, $H_{2}=[q-1]\backslash\{i_{2}\}$, $H_{1}=[k]\cup(I_{1}\cap[q-1])$)
\begin{multline}
\left(\cap_{i\in[k],i\notin I_{1}}(A_{i}^{[k]})^{c}\right)\cap\left(\cap_{k+1\leq i\leq q-1,i\notin I_{1}}A_{i}^{[q-1]}\right)\cap L_{I_{1}}^{k}=\\
\left(\cap_{i\in[k],i\notin I_{1}}(A_{i}^{[k]})^{c}\right)\cap\left(\cap_{k+1\leq i\leq q-1,i\notin(I_{1}\cup\{i_{2}\})}A_{i}^{[q-1]\backslash\{i_{2}\}}\right)\cap A_{i_{2}}^{[q-1]}\cap L_{I_{1}}^{k}\,.\label{eq:ensembles-04}
\end{multline}
 $\widetilde{\widetilde{Z}}_{0:T}^{i_{2}}$ is independent of $Z_{0:T}^{k,1}$,
\ldots, $Z_{0:T}^{k,k}$, $(\widetilde{\widetilde{Z}}_{0:T}^{j})_{j\in\{k+1,\dots,q\}\backslash\{i_{2}\}}$,
$\1_{L_{I_{1}}^{k}}$, $(\1_{(A_{j}^{[k]})^{c}})_{1\leq j\leq k,j\notin I_{1}}$,\\
 $(\1_{A_{j}^{[q-1]\backslash\{i_{2}\}}})_{k+1\leq j\leq q-1,j\notin\{I_{1}\cup\{i_{2}\}\}}$.
So, for $I_{1}$ such that $i_{2}\geq k+1$, we obtain (using that
$F\in\mathcal{B}_{0}^{sym}(q)\mbox{)}$ and (\ref{eq:ensembles-03}),
(\ref{eq:ensembles-04})) 
\begin{multline*}
\E(F(Z_{0:T}^{k,1},\dots,Z_{0:T}^{k,k},\widetilde{\widetilde{Z}}_{0:T}^{k+1},\dots,\widetilde{\widetilde{Z}}_{0:T}^{q})\prod_{i\in[k]}\1_{(A_{i}^{[k]})^{c}}\prod_{k+1\leq i\leq q-1}\1_{A_{i}^{[q-1]}}\times\1_{A_{q}^{c}}\1_{L_{I_{1}}})=\\
\E(F(\dots)\prod_{i\in[k],i\notin I_{1}}\1_{(A_{i}^{[k]})^{c}}\prod_{k+1\leq i\leq q-1\,,\,i\notin I_{1}\cup\{i_{2}\}}\1_{A_{i}^{[q-1]\backslash\{i_{2}\}}}\times(1-\1_{(A_{i_{2}}^{[q-1]})^{c}})\1_{L_{I_{1}}^{k}})=\\
-\E(F(\dots)\prod_{i\in[k],i\notin I_{1}}\1_{(A_{i}^{[k]})^{c}}\prod_{k+1\leq i\leq q-1\,,\,i\notin I_{1}\cup\{i_{2}\}}\1_{A_{i}^{[q-1]\backslash\{i_{2}\}}}\times\1_{(A_{i_{2}}^{[q-1]})^{c}}\1_{L_{I_{1}}^{k}})\,.
\end{multline*}

For $I_{1}$ such that $i_{2}\in[k]$, we have (using (\ref{eq:ensembles-03}))
\begin{multline*}
\E(F(Z_{0:T}^{k,1},\dots,Z_{0:T}^{k,k},\widetilde{\widetilde{Z}}_{0:T}^{k+1},\dots,\widetilde{\widetilde{Z}}_{0:T}^{q})\prod_{i\in[k]}\1_{(A_{i}^{[k]})^{c}}\prod_{k+1\leq i\leq q-1}\1_{A_{i}^{[q-1]}}\times\1_{A_{q}^{c}}\1_{L_{I_{1}}})\\
=\E(F(\dots)\prod_{i\in[k],i\notin I_{1}}\1_{(A_{i}^{[k]})^{c}}\times\1_{L_{I_{1}}^{k}})\,.
\end{multline*}
Recall $\mathcal{P}_{m}$ defined in Section \ref{sub:Technical-lemmas}.
Starting from (\ref{eq:ensembles-05}), (\ref{eq:ensembles-05bis}),
(\ref{eq:ensembles-02}), and proceeding recursively, we obtain that
\begin{multline}
\E(F(Z_{0:T}^{1},\dots,Z_{0:T}^{q})\1_{\left(A_{1}\cap\dots\cap A_{q}\right)^{c}})=\\
\sum_{k=1}^{q}\sum_{m=1}^{q}\left(\begin{array}{c}
q\\
k
\end{array}\right)\sum_{(I_{1},\dots,I_{m})\in\mathcal{P}_{m}}(-1)^{s(I_{1,}\dots,I_{m})}\times\E(F(Z_{0:T}^{k,1},\dots,Z_{0:T}^{k,k},\widetilde{\widetilde{Z}}_{0:T}^{k+1},\dots,\widetilde{\widetilde{Z}}_{0:T}^{q})\prod_{j\in[m]}\1_{L_{I_{j}}^{k}})\,,\label{eq:dec-02}
\end{multline}
where $s(I_{1},\dots,I_{m})=\#\{j\in[m],\max(I_{j})\geq k+1\}$. And
in the same way 
\begin{multline}
\E(F(\widetilde{\widetilde{Z}}_{0:T}^{1},\dots,\widetilde{\widetilde{Z}}_{0:T}^{q})\1_{\left(A_{1}\cap\dots\cap A_{q}\right)^{c}})=\\
\sum_{k=1}^{q}\sum_{m=1}^{q}\left(\begin{array}{c}
q\\
k
\end{array}\right)\sum_{(I_{1},\dots,I_{m})\in\mathcal{P}_{m}}(-1)^{s(I_{1,}\dots,I_{m})}\times\E(F(\widetilde{\widetilde{Z}}_{0:T}^{1},\dots,\widetilde{\widetilde{Z}}_{0:T}^{q})\prod_{j\in[m]}\1_{L_{I_{j}}^{k}})\,,\label{eq:dec-03}
\end{multline}
By Lemma \ref{Lem:qlim}, we then obtain that, for $q$ even, using
the symmetry of the problem (recall the definition of $J_{k}$ from
(\ref{eq:def-J-k})), 
\begin{multline*}
N^{q/2}\E(F(Z_{0:T}^{1},\dots,Z_{0:T}^{q})\1_{\left(A_{1}\cap\dots\cap A_{q}\right)^{c}})\convN\\
\sum_{1\leq k\leq q\,,\,k\text{ even}}(-1)^{\frac{q-k}{2}}\left(\begin{array}{c}
q\\
k
\end{array}\right)J_{k}J_{q-k}\\
\times\E(\E_{\widetilde{\mathcal{K}}_{T}}(F(\widetilde{Z}_{0:T}^{1},\dots,\widetilde{Z}_{0:T}^{k},\widetilde{\widetilde{Z}}_{0:T}^{k+1},\dots,\widetilde{\widetilde{Z}}_{0:T}^{q})|\widetilde{L}_{1,q})\prod_{1\leq i\leq q/2}\int_{0}^{T}\Lambda\widetilde{K}_{s}^{2i-1}\widetilde{K}_{s}^{2i}ds)\,,
\end{multline*}
\begin{multline*}
N^{q/2}\E(F(\widetilde{\widetilde{Z}}_{0:T}^{1},\dots,\widetilde{\widetilde{Z}}_{0:T}^{q})\1_{\left(A_{1}\cap\dots\cap A_{q}\right)^{c}})\convN\\
\sum_{1\leq k\leq q,\,k\mbox{ even}}(-1)^{\frac{q-k}{2}}\left(\begin{array}{c}
q\\
k
\end{array}\right)J_{k}J_{q-k}\\
\times\E(\E_{\widetilde{\mathcal{K}}_{T}}(F(\widetilde{\widetilde{Z}}_{0:T}^{1},\dots,\widetilde{\widetilde{Z}}_{0:T}^{q})|\widetilde{L}_{1,q})\prod_{1\leq i\leq q/2}\int_{0}^{T}\Lambda\widetilde{K}_{s}^{2i-1}\widetilde{K}_{s}^{2i}ds)\,,
\end{multline*}
and for $q$ odd, these limits are $0$. We then use the equality
$\left(\begin{array}{c}
q\\
k
\end{array}\right)J_{k}J_{q-k}=\left(\begin{array}{c}
q/2\\
k/2
\end{array}\right)J_{q}$ to finish the proof.

\end{proof}

\begin{corollary}\label{cor:terme-k-borne}For $F\in\mathcal{B}_{0}^{sym}(q)$,
we have
\[
\left|\E((\eta_{0:T}^{N})^{\odot q}(F))\right|\leq\frac{2^{2q+1}(\Lambda T\vee1)^{q+1}(3q)!}{(e^{-\Lambda T})^{2q+1}(1-e^{-\Lambda T})q!(q-1)!}\frac{\Vert F\Vert_{\infty}}{(N-1)^{q/2}}\left(\frac{1}{\left(\left\lceil \frac{q}{4}\right\rceil \right)!}+\frac{1}{(N-1)^{\frac{q}{4}}}\right)\,.
\]

\end{corollary}

\begin{proof}

For $m,m'\in[q]$ and $(I_{1}\dots,I_{m})\in\mathcal{P}_{m}$, $(I'_{1}\dots,I'_{m'})\in\mathcal{P}_{m'}$,
we have that $(\cap_{j\in[m]}L_{I_{j}})\cap(\cap_{j\in[m']}L_{I'_{j}})=\emptyset$
if $(I_{1}\dots,I_{m})\neq(I'_{1}\dots,I'_{m'})$. Note, also, that
for all $m,k\in[q]$, $(I_{1}\dots,I_{m})\in\mathcal{P}_{m}$, we
have that $\cap_{j\in[m]}L_{I_{j}}^{k}\subset\left(B\cup\{\widetilde{L}_{T}\geq\left\lceil q/2\right\rceil \}\right)$
(by (\ref{eq:inclusion-chantme}) and the inequality $\frac{q}{2}+1\geq\left\lceil \frac{q}{2}\right\rceil $).
Hence, by (\ref{eq:dec-02}) and Lemma \ref{lem:ens-negligeable},
\begin{multline*}
\left|\E(F(Z_{0:T}^{1},\dots,Z_{0:T}^{q})\1_{(A_{1}\cap\dots\cap A_{q})^{c}})\right|\\
\leq\sum_{k=1}^{q}\left(\begin{array}{c}
q\\
k
\end{array}\right)\p(B\cup\{\widetilde{L}_{T}\geq\left\lceil q/2\right\rceil \})\Vert F\Vert_{\infty}\\
\leq2^{q}\p(B\cup\{\widetilde{L}_{T}\geq\left\lceil q/2\right\rceil \})\Vert F\Vert_{\infty}\\
\leq2^{q}\sum_{r=0}^{q}\left(\begin{array}{c}
q\\
r
\end{array}\right)\p(E_{T}^{1}\geq1,\dots,E_{T}^{r}\geq1,\widetilde{L}_{T}\geq\left\lceil \frac{q-r}{2}\right\rceil )\Vert F\Vert_{\infty}\\
\mbox{(note that }\forall r\in\{0,\dots,q\},\,\mbox{using the notation }l=\left\lceil \frac{q-r}{2}\right\rceil \mbox{: }\,2r+2l\leq2q+1\,,\,r+l\leq q+1,\\
\frac{(q+2r+2l-1)!}{(q-1)!r^{r}l!}\leq\frac{(2q+r)!}{(q-1)!r!l!}\leq\frac{(3q)!}{(q-1)!q!}\mbox{ ) }\\
\leq2^{q}\sum_{r=0}^{q}\left(\begin{array}{c}
q\\
r
\end{array}\right)\frac{(\Lambda T\vee1)^{q+1}(3q)!\Vert F\Vert_{\infty}}{(N-1)^{r+\left\lceil \frac{q-r}{2}\right\rceil }(e^{-\Lambda T})^{2q+1}(1-e^{-\Lambda T})(q-1)!q!\left(\left\lceil \frac{q-r}{2}\right\rceil \right)!}\\
\leq2^{q}\frac{(\Lambda T\vee1)^{q+1}(3q)!\Vert F\Vert_{\infty}}{(e^{-\Lambda T})^{2q+1}(1-e^{-\Lambda T})(q-1)!q!}\left(\sum_{r=0}^{\lfloor q/2\rfloor}\left(\begin{array}{c}
q\\
r
\end{array}\right)\frac{1}{(N-1)^{q/2}\left(\left\lceil \frac{q}{4}\right\rceil \right)!}\right.\\
\left.+\sum_{r=\lfloor q/2\rfloor+1}^{q}\left(\begin{array}{c}
q\\
r
\end{array}\right)\frac{1}{(N-1)^{q/2}(N-1)^{\frac{q}{4}}}\right)\\
\leq\frac{2^{2q}(\Lambda T\vee1)^{q+1}(3q)!}{(e^{-\Lambda T})^{2q+1}(1-e^{-\Lambda T})(q-1)!q!}\frac{\Vert F\Vert_{\infty}}{(N-1)^{q/2}}\left(\frac{1}{\left(\left\lceil \frac{q}{4}\right\rceil \right)!}+\frac{1}{(N-1)^{\frac{q}{4}}}\right)\,.
\end{multline*}
The same is true if we replace the $Z^{i}$'s by $\widetilde{\widetilde{Z}}^{i}$'s.
Thus (\ref{eq:dec-01}) gives us the desired bound.

\end{proof}

\subsection{Wick formula}

We suppose here that $q$ is even. 

\begin{definition}

We introduce an auxiliary infinite system of particles $(\check{Z}_{0:T}^{1},\check{Z}_{0:T}^{2},\dots)$
such that
\begin{itemize}
\item it has interaction times $\left\{ T-\widetilde{T}_{k},k\geq1\right\} \cap\left\{ T-t,t\in\mathcal{T}^{\{1,2\}}\cup\dots\cup\mathcal{T}^{\{q-1,q\}}\right\} $ 
\item the rest of the definition is the same as for $(\widetilde{Z}_{0:T}^{i})_{i\geq1}$. 
\end{itemize}
By doing this, we have coupled the interaction times of the system
$(\check{Z}_{0:T}^{i})_{i\geq1}$ and of the system $(\widetilde{Z}_{0:T}^{i})_{i\geq1}$.
We can couple further and assume that $(\check{Z}_{0:T}^{1},\dots,\check{Z}_{0:T}^{q})$
and $(\widetilde{\widetilde{Z}}_{0:T}^{1},\dots,\widetilde{\widetilde{Z}}_{0:T}^{q})$
coincide on the event $ $$\left\{ \widetilde{T}_{k},k\geq1\right\} \cap\left(\mathcal{T}^{\{1,2\}}\cup\dots\cup\mathcal{T}^{\{q-1,q\}}\right)=\left\{ \widetilde{T}_{k},k\geq1\right\} $.

\end{definition}

The system $(\check{Z}_{0:T}^{i})_{1\leq i}$ is obtained from $(\widetilde{Z}_{0:T}^{i})_{1\leq i\leq k}$
by stripping off the links which are not internal to $\widetilde{C}^{1},\widetilde{C}^{2}$
or $\widetilde{C}^{3},\widetilde{C}^{4}$, \ldots . 

We set for all $f,g$ bounded $\mathbb{D}(\R_{+},\R^{d})\rightarrow\R$
\begin{equation}
V_{0:T}(f,g)=\E\left(\E(f(\check{Z}_{0:T}^{1})g(\check{Z}_{0:T}^{2})-f(\widetilde{\widetilde{Z}}_{0:T}^{1})g(\widetilde{\widetilde{Z}}_{0:T}^{2})|\widetilde{L}_{1,2},\widetilde{K}_{0:T}^{1},\widetilde{K}_{0:T}^{2})\int_{0}^{T}\Lambda\widetilde{K}_{s}^{1}\widetilde{K}_{s}^{2}ds\right)\,.\label{eq:def-V-0-T}
\end{equation}
 Note that for all $f,g$, $V_{0:T}(f,g)=V_{0:T}(g,f)$. Note that
the formula above cannot be simplified with the use of the tower formula.
For all $k\in\N^{*}$, we set $\mathcal{I}_{k}$ to be the set of
partitions of $[k]$ into subsets of cardinality $2$.

\begin{corollary}{[}Wick formula{]} \label{Cor:wickproduct} For
$F\in\Bsym(q)$ of the form $F=(f_{1}\otimes\dots\otimes f_{q})_{\text{sym}}$
and $q$ even, 
\[
N^{q/2}\E(F(Z_{0:T}^{1},\dots,Z_{0:T}^{q}))\convN\sum_{J\in\mathcal{I}_{q}}\prod_{\{a,b\}\in J}V_{0:T}(f_{a},f_{b})\,.
\]
 \end{corollary}

The name ``Wick formula'' comes from the Wick formula on the expectation
of a product of Gaussians. In this formula, there is a sum over pairings,
just as in the above Corollary. See Theorem 22.3, p. 360 in \cite{nica-speicher-2006}
for the Wick formula.

\begin{proof}To shorten the notations, we will write: 
\[
\prod_{i=1}^{q/2}\int_{0}^{T}\Lambda\widetilde{K}_{s}^{2i-1}\widetilde{K}_{s}^{2i}ds=p\,.
\]
 With this particular form for $F$, the limit in (\ref{eq:conWick01})
of Proposition \ref{Prop:convq2} becomes
\begin{multline*}
\frac{1}{q!}\sum_{\sigma\in\mathcal{S}_{q}}\sum_{k=1}^{q/2}J_{q}\left(\begin{array}{c}
q/2\\
k
\end{array}\right)(-1)^{\frac{q}{2}-k}\\
\times\E(\E_{\widetilde{\mathcal{K}}_{T}}(f_{\sigma(1)}(\widetilde{Z}_{0:T}^{1})\dots f_{\sigma(2k)}(\widetilde{Z}_{0:T}^{2k})f_{\sigma(2k+1)}(\widetilde{\widetilde{Z}}_{0:T}^{2k+1})\dots f_{\sigma(q)}(\widetilde{\widetilde{Z}}_{0:T}^{q})\\
-f_{\sigma(1)}(\widetilde{\widetilde{Z}}_{0:T}^{1})\dots f_{\sigma(q)}(\widetilde{\widetilde{Z}}_{0:T}^{q})|\widetilde{L}_{1,q})p)\\
=\frac{1}{q!}\sum_{\sigma\in\mathcal{S}_{q}}\sum_{k=0}^{q/2}J_{q}\left(\begin{array}{c}
q/2\\
k
\end{array}\right)(-1)^{\frac{q}{2}-k}\\
\times\E(\E_{\widetilde{\mathcal{K}}_{T}}(f_{\sigma(1)}(\widetilde{Z}_{0:T}^{1})\dots f_{\sigma(2k)}(\widetilde{Z}_{0:T}^{2k})f_{\sigma(2k+1)}(\widetilde{\widetilde{Z}}_{0:T}^{2k+1})\dots f_{\sigma(q)}(\widetilde{\widetilde{Z}}_{0:T}^{q})|\widetilde{L}_{1,q})p)\,,
\end{multline*}
because $\sum_{k=1}^{q/2}\left(\begin{array}{c}
q/2\\
k
\end{array}\right)(-1)^{\frac{q}{2}-k}=(-1)^{\frac{q}{2}+1}$. Using the exchangeability property of the particle systems, we
can transform the last expression into 
\begin{multline}
\frac{1}{q!}\sum_{\sigma\in\mathcal{S}_{q}}\sum_{k=0}^{q/2}J_{q}(-1)^{\frac{q}{2}-k}\sum_{I\subset[q/2],\#I=k}\E(\E_{\widetilde{\mathcal{K}}_{T}}(\prod_{i\in I}f_{\sigma(2i-1)}(\widetilde{Z}_{0:T}^{2i-1})f_{\sigma(2i)}(\widetilde{Z}_{0:T}^{2i})\\
\prod_{i\notin I}f_{\sigma(2i-1)}(\widetilde{\widetilde{Z}}_{0:T}^{2i-1})f_{\sigma(2i)}(\widetilde{\widetilde{Z}}_{0:T}^{2i})|\widetilde{L}_{1,q})p)=\\
\frac{1}{q!}\sum_{\sigma\in\mathcal{S}_{q}}J_{q}\E(\E_{\widetilde{\mathcal{K}}_{T}}(\prod_{i=1}^{q/2}(f_{\sigma(2i-1)}(\widetilde{Z}_{0:T}^{2i-1})f_{\sigma(2i)}(\widetilde{Z}_{0:T}^{2i})-f_{\sigma(2i-1)}(\widetilde{\widetilde{Z}}_{0:T}^{2i-1})f_{\sigma(2i)}(\widetilde{\widetilde{Z}}_{0:T}^{2i}))|\widetilde{L}_{1,q})p)=\\
\frac{1}{q!}\sum_{\sigma\in\mathcal{S}_{q}}J_{q}\E(\E_{\widetilde{\mathcal{K}}_{T}}(\prod_{i=1}^{q/2}(f_{\sigma(2i-1)}(\check{Z}_{0:T}^{2i-1})f_{\sigma(2i)}(\check{Z}_{0:T}^{2i})-f_{\sigma(2i-1)}(\widetilde{\widetilde{Z}}_{0:T}^{2i-1})f_{\sigma(2i)}(\widetilde{\widetilde{Z}}_{0:T}^{2i}))\\
|\widetilde{L}_{1,2},\dots,\widetilde{L}_{q-1,q})p)\,,\label{eq:last-above}
\end{multline}
the last equality being true because 
\[
\mathcal{L}((\widetilde{Z}_{0:T}^{i})_{i\geq1},(\widetilde{\widetilde{Z}}_{0:T}^{i})_{i\geq1}|\widetilde{\mathcal{K}}_{T},\widetilde{L}_{1,q})=\mathcal{L}((\check{Z}_{0:T}^{i})_{i\geq1},(\widetilde{\widetilde{Z}}_{0:T}^{i})_{i\geq1}|\widetilde{\mathcal{K}}_{T},\widetilde{L}_{1,2},\dots,\widetilde{L}_{q-1,q}).
\]
 By Lemma \ref{Lem:law-K-tilde}, \ref{lemKtilde:05}, the processes
$\widetilde{L}^{\{1,2\}},\dots,\widetilde{L}^{\{q,q-1\}}$, defined
in (\ref{eq:def-L-tilde-(j-1,j)}), are independent, conditionally
to $\widetilde{\mathcal{K}}_{T}$. And for all $i\in[q/2]$, 
\begin{multline*}
\E_{\widetilde{\mathcal{K}}_{T}}(f(\check{Z}_{0:T}^{2i-1})f_{\sigma(2i)}(\check{Z}_{0:T}^{2i})-f_{\sigma(2i-1)}(\widetilde{\widetilde{Z}}_{0:T}^{2i-1})f_{\sigma(2i)}(\widetilde{\widetilde{Z}}_{0:T}^{2i})|\widetilde{L}_{1,2},\dots,\widetilde{L}_{q-1,q})\\
=\E(f(\check{Z}_{0:T}^{2i-1})f_{\sigma(2i)}(\check{Z}_{0:T}^{2i})-f_{\sigma(2i-1)}(\widetilde{\widetilde{Z}}_{0:T}^{2i-1})f_{\sigma(2i)}(\widetilde{\widetilde{Z}}_{0:T}^{2i})|\widetilde{L}_{2i-1,2i},\widetilde{K}_{0:T}^{2i-1},\widetilde{K}_{0:T}^{2i})\,.
\end{multline*}
 So, the quantity in (\ref{eq:last-above}) is equal to 
\begin{multline*}
\frac{1}{q!}\sum_{\sigma\in\mathcal{S}_{q}}J_{q}\E(\prod_{i=1}^{q/2}[\E((f_{\sigma(2i-1)}(\check{Z}_{0:T}^{2i-1})f_{\sigma(2i)}(\check{Z}_{0:T}^{2i})\\
-f_{\sigma(2i-1)}(\widetilde{\widetilde{Z}}_{0:T}^{2i-1})f_{\sigma(2i)}(\widetilde{\widetilde{Z}}_{0:T}^{2i}))|\widetilde{L}_{2i-1,2i},\widetilde{K}_{0:T}^{2i-1},\widetilde{K}_{0:T}^{2i})\times\int_{0}^{T}\Lambda\widetilde{K}_{s}^{2i-1}\widetilde{K}_{s}^{2i-1}ds)])=\\
\mbox{(by Lemma \ref{Lem:law-K-tilde}, 4)}\\
\frac{1}{q!}\sum_{\sigma\in\mathcal{S}_{q}}J_{q}\prod_{i=1}^{q/2}\E(\E((f_{\sigma(2i-1)}(\check{Z}_{0:T}^{2i-1})f_{\sigma(2i)}(\check{Z}_{0:T}^{2i})\\
-f_{\sigma(2i-1)}(\widetilde{\widetilde{Z}}_{0:T}^{2i-1})f_{\sigma(2i)}(\widetilde{\widetilde{Z}}_{0:T}^{2i}))|\widetilde{L}_{2i-1,2i},\widetilde{K}_{0:T}^{2i-1},\widetilde{K}_{0:T}^{2i})\int_{0}^{T}\Lambda\widetilde{K}_{s}^{2i-1}\widetilde{K}_{s}^{2i-1}ds)=\\
\sum_{J\in\mathcal{I}_{q}}\prod_{\{a,b\}\in J}V_{0:T}(f_{a},f_{b})\,.
\end{multline*}

\end{proof}

\section{Proof of convergence theorems\label{Sec:conv}}

\subsection{Proof of Theorem \ref{Th:conv-ps-debut} (almost sure convergence)\label{sub:Almost-sure-convergence}}

\begin{proof}

We recall the notations of \cite{del-moral-patras-rubenthaler-2008}.
For any empirical measure $m(x)=\frac{1}{N}\sum_{i=1}^{N}\delta_{x^{i}}$
(based on $N$ points $x^{1},x^{2},\dots,x^{N}$), 

\[
m(x)^{\otimes q}:=\frac{1}{N^{q}}\sum_{a\in[N]^{[q]}}\delta_{(x^{a(1)},\dotsm,x^{a(q)})}\ ,
\]
 where $[N]^{[q]}=\left\{ a:[q]\rightarrow[N]\right\} $. Note that
for any $F$,
\[
m(x)^{\otimes q}(F)=m(x)^{\otimes q}(F_{\text{sym}})\ .
\]
 We define, for all $p\in[q]$, $[q]_{p}^{[q]}:=\{a\in[q]^{[q]},\#\text{Im}(a)=p\}$,
and ($\forall k\leq q$),

\[
\partial^{k}L_{q}=\sum_{q-k\leq p\leq q}s(p,q-k)\frac{1}{(q)_{p}}\sum_{a\in[q]_{p}^{[q]}}a\,,
\]
 (the $s(.,.)$ are the Stirling numbers of the first kind), and for
all $F$ (of $q$ variables), for all $b\in[q]^{[q]}$,
\[
D_{b}(F)(x^{1},\dots,x^{q})=F(x^{b(1)},\dots,x^{b(q)})\ ,
\]
 
\[
D_{\partial^{k}L_{q}}(F)=\sum_{q-k\leq p\leq q}s(p,q-k)\frac{1}{(q)_{p}}\sum_{a\in[q]_{p}^{[q]}}D_{a}(F)\ .
\]
 The derivative-like notation $\partial^{k}L$ comes from \cite{del-moral-patras-rubenthaler-2008},
where it makes sense to think of a derivative at this point. We keep
the same notation in order to be consistent, but it has no particular
meaning in our setting. We then have, by Corollary 2.3 p. 789 of \cite{del-moral-patras-rubenthaler-2008},
for any empirical measure $m(x)$ (based on $N$ points) and for any
$F$ of $q$ variables,

\[
m(x)^{\otimes q}(F)=m(x)^{\odot q}\left(\sum_{0\leq k<q}\frac{1}{N^{k}}D_{\partial^{k}L_{q}}(F)\right)\ .
\]
 Suppose $F\in\mathcal{B}_{0}^{sym}(q)$, we then obtain

\begin{equation}
\E\left((\eta_{0:T}^{N})^{\otimes q}(F)\right)=\sum_{0\leq k<q}\frac{1}{N^{k}}\sum_{q-k\leq p\leq q}s(p,q-k)\sum_{a\in[q]_{p}^{[q]}}\E\left((\eta_{0:T}^{N})^{\odot q}(D_{a}(F))\right)\ .\label{eq:croix-point}
\end{equation}
 We take $a\in[q]^{[q]}$ with $p=\#\mbox{Im}(a)\geq q-k$ and $k<q/2$.
Note that $\#\{i\in[q],\#a^{-1}(\{i\})=1\}\geq q-2k>0$.

We have now to use the Hoeffding's decomposition (see \cite{delapena-gine-1999,lee-1990},
or \cite{del-moral-patras-rubenthaler-2009}, Section 4, for the details).
For any symmetrical $G:\mathbb{D}([0,T],(\R^{d})^{q})\rightarrow\R$,
we define 
\[
\theta=\int G(x_{1},\dots x_{q})\widetilde{P}_{0:T}(dx_{1},\dots,dx_{q})\,,
\]
 
\[
G^{(j)}(x_{1},\dots,x_{j})=\int G(x_{1},\dots,x_{q})\widetilde{P}_{0:T}^{\otimes(q-j)}(dx_{j+1},\dots,dx_{q})\,,
\]
 and recursively
\[
h^{(1)}(x_{1})=G^{(1)}(x_{1})-\theta\,,
\]
 
\[
h^{(k)}(x_{1},\dots,x_{k})=G^{(j)}(x_{1},\dots,x_{k})-\sum_{i=1}^{j-1}\sum_{(j,i)}h^{(i)}-\theta\,,
\]
 where $\sum_{(j,i)}h^{(i)}$ is an abbreviation for the function
\[
(x_{1},\dots,x_{j})\mapsto\sum_{1\leq r_{1}<\dots<r_{i}\leq j}h^{(i)}(x_{r_{1}},\dots,x_{r_{i}})\,.
\]
 For all $j$, $h^{(j)}$ is in $\mathcal{B}_{0}^{sym}(j)$. We have
the formula 
\[
G(x_{1},\dots,x_{q})=h^{(q)}(x_{1},\dots,x_{q})+\sum_{j=1}^{q-1}\sum_{(q,j)}h^{(j)}\,.
\]

We take now $G(x_{1},\dots,x_{q})=D_{a}(F)$ (still with $F\in\mathcal{B}_{0}^{sym}(q)$).
For $j<q-2k$, $G^{(j)}=0$. So we can show by recurrence that $h^{(j)}=0$
for $j<q-2k$. So
\[
G(x_{1},\dots,x_{q})=h^{(q)}(x_{1},\dots,x_{q})+\sum_{j=q-2k}^{q-1}\sum_{(q,j)}h^{(j)}\,.
\]
 So, by Corollary \ref{cor:terme-k-borne}, we have for some constant
$C$ 
\[
\E(D_{a}F(Z_{0:T}^{1},\dots,Z_{0:T}^{q}))\leq\frac{C}{N^{(q-2k)/2}}\,.
\]
 And so, by (\ref{eq:croix-point}),
\[
\E\left((\eta_{0:T}^{N})^{\otimes q}(F)\right)\leq\frac{C}{N^{\frac{q}{2}}}\ .
\]
 Suppose that we take a bounded function $f:\mathbb{D}([0,T],\R^{d})\rightarrow\R$.
We set $\bar{f}=f-\widetilde{P}_{0:T}(f)$. We then have (with the
notation $\bar{f}^{\otimes q}(x^{1},\dots,x^{q}):=\bar{f}(x^{1})\times\dots\times\bar{f}(x^{q})$)
\begin{eqnarray}
\E(((\eta_{0:T}^{N}(f)-\widetilde{P}_{0:T}(f))^{q}) & = & \E((\eta_{0:T}^{N}(\bar{f}))^{q})\nonumber \\
 & = & \E((\eta_{0:T}^{N})^{\otimes q}(\bar{f}^{\otimes q}))\nonumber \\
 & = & \E((\eta_{0:T}^{N})^{\otimes q}(\overline{f}^{\otimes q})_{\mbox{sym }})\nonumber \\
 & \leq & \frac{C}{N^{\frac{q}{2}}}\ .\label{eq:conv-Lq}
\end{eqnarray}
 Provided we take $q=4$, we can apply Borel-Cantelli Lemma to finish
the proof.

\end{proof}

\subsection{Proof of Theorem \ref{Th:TCL-debut} (central-limit theorem)\label{sub:Central-limit-theorem}}

\begin{proof} To simplify, we suppose here that $\Vert f_{1}\Vert_{\infty}\leq1$,
\ldots{}, $\Vert f_{q}\Vert_{\infty}\leq1$. For any $u_{1},\dots,u_{q}\in\R$,
we have:

\begin{eqnarray}
 &  & \E\left(\exp\left(N\eta_{0:T}^{N}\left(\log\left(1+\frac{iu_{1}f_{1}+\dots+iu_{q}f_{q}}{\sqrt{N}}\right)\right)\right)\right)\label{eq:debut-du-calcul}\\
 &  & ~~~=\E\left(\prod_{j=1}^{N}\left(1+\frac{iu_{1}f_{1}(Z_{0:T}^{j})+\dots+iu_{q}f_{q}(Z_{0:T}^{j})}{\sqrt{N}}\right)\right)\nonumber \\
 &  & ~~~=\E(\sum_{0\leq k\leq N}\frac{1}{N^{k/2}}\sum_{1\leq j_{1},\dots,j_{k}\leq q}i^{k}u_{j_{1}}\dots u_{j_{k}}\times\sum_{1\leq i_{1}<\dots<i_{k}\leq N}f_{j_{1}}(Z_{0:T}^{i_{1}})\dots f_{j_{k}}(Z_{0:T}^{i_{k}}))\nonumber \\
 &  & ~~~=\sum_{0\leq k\leq N}\frac{(N)_{k}}{N^{k/2}}\sum_{1\leq j_{1},\dots,j_{k}\leq q}i^{k}u_{j_{1}}\dots u_{j_{k}}\frac{1}{k!}\E\left((\eta_{0:T}^{N})^{\odot k}(f_{j_{1}}\otimes\dots\otimes f_{j_{k}})\right)\nonumber \\
 &  & ~~~=\sum_{0\leq k\leq N}\frac{(N)_{k}}{N^{k/2}}\sum_{1\leq j_{1},\dots,j_{k}\leq q}i^{k}u_{j_{1}}\dots u_{j_{k}}\frac{1}{k!}\E\left((\eta_{0:T}^{N})^{\odot k}(f_{j_{1}}\otimes\dots\otimes f_{j_{k}})_{sym}\right)\,.\nonumber 
\end{eqnarray}
By Corollary \ref{cor:terme-k-borne}, for all $k\in[N]$, we have
(computing very roughly)
\begin{multline*}
\left|\frac{(N)_{k}}{N^{k/2}}\sum_{1\leq j_{1},\dots,j_{k}\leq q}i^{k}u_{j_{1}}\dots u_{j_{k}}\frac{1}{k!}\E\left((\eta_{0:T}^{N})^{\odot k}(f_{j_{1}}\otimes\dots\otimes f_{j_{k}})_{sym}\right)\right|\\
\leq\frac{q^{k}(\max(|u_{1}|,\dots,|u_{q}|))^{k}}{k!}\times\frac{2^{2k+1}(3k)!(\Lambda T\vee1)^{k+1}}{(k-1)!k!(e^{-\Lambda T})^{2k+1}(1-e^{-\Lambda T})}\\
\times\left(\frac{1}{\left(\left\lceil \frac{k}{4}\right\rceil \right)!}+\frac{1}{(N-1)^{\frac{k}{4}}}\right)\frac{(N)_{k}}{(N-1)^{k/2}N^{k/2}}\\
\leq\frac{q^{k}2^{2k+1}3^{3k}(3k)(\max(|u_{1}|,\dots,|u_{q}|))^{k}(\Lambda T\vee1)^{k+1}}{(e^{-\Lambda T})^{2k+1}(1-e^{-\Lambda T})}\left(\frac{1}{\left(\left\lceil \frac{k}{4}\right\rceil \right)!}+\frac{1}{(N-1)^{\frac{k}{4}}}\right)\,,
\end{multline*}
and this last term is summable in $k$ if $N$ is big enough. Using
Corollary \ref{Cor:wickproduct} and Proposition \ref{Prop:convq2},
we then obtain: 
\begin{multline}
\E\left(\exp\left(N\eta_{0:T}^{N}\left(\log\left(1+\frac{iu_{1}f_{1}+\dots iu_{q}f_{q}}{\sqrt{N}}\right)\right)\right)\right)\\
\underset{N\rightarrow+\infty}{\longrightarrow}\sum_{k\geq0,k\text{ even }}(-1)^{k/2}\sum_{1\leq j_{1},\dots,j_{k}\leq q}\frac{u_{j_{1}}\dots u_{j_{k}}}{k!}\sum\limits _{I_{k}\in{\cal I}_{k}}\prod\limits _{\{a,b\}\in I_{k}}V_{0:T}(f_{j_{a}},f_{j_{b}})\\
=\sum_{k\geq0,k\text{ even }}\frac{(-1)^{k/2}}{2^{k/2}(k/2)!}\times\sum_{1\leq j_{1},\dots,j_{k}\leq q}u_{j_{1}}\dots u_{j_{k}}V_{0:T}^ {}(f_{j_{1}},f_{j_{2}})\dots V_{0:T}(f_{j_{k-1}},f_{j_{k}})\\
=\sum_{k\geq0,k\text{ even }}\frac{(-1)^{k/2}}{2^{k/2}(k/2)!}\left(\sum_{1\leq j_{1},j_{2}\leq q}u_{j_{1}}u_{j_{2}}V_{0:T}^ {}(f_{j_{1}},f_{j_{2}})\right)^{k/2}\\
=\exp\left(-\frac{1}{2}\sum_{1\leq j_{1},j_{2}\leq q}u_{j_{1}}u_{j_{2}}V_{0:T}^ {}(f_{j_{1}},f_{j_{2}})\right)\label{eq:calcul-01}
\end{multline}
We can also write a series development of the $\log$ in (\ref{eq:debut-du-calcul})
and obtain:
\begin{multline}
\E\left(\exp\left(N\eta_{0:T}^{N}(\log(1+\frac{iu_{1}f_{1}+\dots+iu_{q}f_{q}}{\sqrt{N}}))\right)\right)\\
=\E\left(\exp\left(\sum_{k\geq1}\frac{(-1)^{k+1}}{k}N^{1-k/2}\eta_{0:T}^{N}((iu_{1}f_{1}+\dots+iu_{q}f_{q})^{k})\right)\right)\,.\label{eq:calcul-02}
\end{multline}
 We have, for all $k$, $\Vert\frac{(iu_{1}f_{1}+\dots+iu_{q}f_{q})^{k}}{N^{k/2}}\Vert_{\infty}\leq\frac{(|u_{1}|+\dots+|u_{q}|)^{k}}{N^{k/2}}$
for some constant $C$ (independent of $k$). So, the remaining term
in the series development of the log can be bounded by
\[
\Vert\sum_{k\geq3}\frac{(-1)^{k+1}}{k}N^{1-k/2}(iu_{1}f_{1}+\dots+iu_{q}f_{q})^{k}\Vert_{\infty}\leq\frac{C(|u_{1}|+\dots+|u_{q}|)^{4}}{N}\,,
\]
 for some constant $C$, if $N\geq2(|u_{1}|+\dots+|u_{q}|)^{2}$.
So, the limit when $N\rightarrow+\infty$ of (\ref{eq:calcul-02})
is the same as the limit of 
\[
\E\left(\exp\left(\sqrt{N}(iu_{1}\eta_{0:T}^{N}(f_{1})+\dots+iu_{q}\eta_{0:T}^{N}(f_{q})\right)\exp\left(\frac{1}{2}\eta_{0:T}^{N}((u_{1}f_{1}+\dots+u_{q}f_{q})^{2})\right)\right)\,.
\]
 We have, for some constant $C$ and $f:=iu_{1}f_{1}+\dots+iu_{q}f_{q}$
(recall $x\in\R\Rightarrow|e^{ix}|=1$),
\begin{equation}
\vert\E(e^{\sqrt{N}\eta_{0:T}^{N}(f)}e^{-\frac{1}{2}\eta_{0:T}^{N}(f^{2})})-\E(e^{\sqrt{N}\eta_{0:T}^{N}(f)}e^{-\frac{1}{2}\widetilde{P}_{0:T}(f^{2})})\vert\leq C\E(\left|\widetilde{P}_{0:T}(f^{2})-\eta_{0:T}^{N}(f^{2})\right|)\label{eq:maj-lim}
\end{equation}
So, by Theorem \ref{Th:conv-ps-debut}, the left-hand side of (\ref{eq:maj-lim})
goes to $0$ as $N\rightarrow+\infty$. So
\begin{multline*}
\lim_{N\rightarrow0}\exp\left(\sqrt{N}\eta_{0:T}^{N}\left(\log(1+\frac{iu_{1}f_{1}+\dots+iu_{q}f_{q}}{\sqrt{N}})\right)\right)\\
=\lim_{N\rightarrow0}\E\left(e^{\sqrt{N}\eta_{0:T}^{N}(iu_{1}f_{1}+\dots+iu_{q}f_{q})}e^{\frac{1}{2}\widetilde{P}_{0:T}((u_{1}f_{1}+\dots+iu_{q}f_{1})^{2})}\right)\,,
\end{multline*}
 (meaning that if these limits exist, they are equal), which concludes
the proof with, $\forall i,j$,
\begin{equation}
K(i,j)=\widetilde{P}_{0:T}(f_{i}f_{j})+V_{0:T}(f_{i},f_{j})\,.\label{eq:def-K}
\end{equation}

\end{proof}

Note that we can bound the two terms of rhs above. Take $f_{1},\dots,f_{q}$
as above and such that $\Vert f_{1}\Vert_{\infty}\leq1$, \ldots{},
$\Vert f_{q}\Vert_{\infty}\leq1$$ $. For all $i,j$:
\[
|\widetilde{P}_{0:T}(f_{i}f_{j})|\leq1\,,
\]
 
\begin{eqnarray*}
|V_{0:T}(f_{i},f_{j})| & \leq & 2\E(T\Lambda\widetilde{K}_{T}^{1}\widetilde{K}_{T}^{2})\\
\mbox{ (by Lemma \ref{lem:maj-E-K-tilde}) } & \leq & \frac{2T\Lambda e^{2T\Lambda}(q+1)q}{(1-e^{\Lambda T})^{2}}\,.
\end{eqnarray*}

\subsection{Proof of Corollary \ref{cor:TCL}}

\begin{proof}The result is a consequence of Theorem \ref{Th:TCL-debut}
from this paper and of Theorem 4.1 from \cite{del-moral-patras-rubenthaler-2009}.
We only have to prove that for all $j\geq2$, $f\in\mathcal{B}_{0}^{sym}(j)$,
\[
\E\left(\left((\eta_{0:T}^{N})^{\odot j}(f)\right)^{2}\right)\leq\frac{C}{N^{j}}\,,
\]
for some constant $C$ which may depend on $j$, $f$, $T$. Looking
at the proof of Lemma 4.3 of \cite{del-moral-patras-rubenthaler-2009},
we see that we need only to prove that for all $k\in\{j+1,\dots,2j\}$,
$r\in[k]$, for all $F:\mathbb{D}([0,T],\R^{d})^{k}\rightarrow\R$
bounded measurable, symmetric in the $k-r$ last variables and such
that $\int_{\mathbb{D}([0,T],E)}F(z_{1},\dots,z_{k})\widetilde{P}_{0:T}(dx_{i})=0$,
for all $i\in\{r+1,\dots,k\}$, we have
\[
\left|\E((\eta_{0:T}^{N})^{\odot k}(F))\right|\leq\frac{C}{N^{\frac{k-r}{2}}}\,,
\]
for some constant $C$ depending on $F$, $k$, $r$. The proof of
this inequality follows the outline of the proof of Proposition \ref{Prop:convq2}.
Here, we write only the beginning of the decomposition. We have (for
$F$, $k$, $r$ as above)
\begin{multline*}
\E((\eta_{0:T}^{N})^{\odot k}(F))=\E(F(Z_{0:T}^{1},\dots,Z_{0:T}^{k}))\\
=\E(F(Z_{0:T}^{1},\dots,Z_{0:T}^{k})\1_{\left(A_{r+1}\cap\dots\cap A_{k}\right)^{c}})+\E(F(Z_{0:T}^{1},\dots,Z_{0:T}^{k})\1_{A_{r+1}\cap\dots\cap A_{k}})\\
=\E(F(Z_{0:T}^{1},\dots,Z_{0:T}^{k})\1_{\left(A_{r+1}\cap\dots\cap A_{k}\right)^{c}})-\E(F(Z_{0:T}^{1},\dots,Z_{0:T}^{r},\widetilde{\widetilde{Z}}_{0:T}^{r+1},\dots,\widetilde{\widetilde{Z}}_{0:T}^{k})\1_{\left(A_{r+1}\cap\dots\cap A_{k}\right)^{c}})\,.
\end{multline*}

\end{proof}

\section{Appendix}

\subsection{Proof of Lemma\label{sub:Proof-of-Lemma} \ref{Lem:equlaw1} }

\begin{proof}

Let us here give a brief explanation of why this equality is true.
We start the construction of the link times at time $0$. We first
look at the times $\{T_{k},k\geq1\}\cap\{T'_{l},l\geq1\}$. The law
of $\tau=\inf\{T_{k},k\geq1\}\cap\{T'_{l},l\geq1\}$ is $\mathcal{E}(\frac{\Lambda q(N-q)_{+}}{N-1})$
. At time $\tau$, we choose $r(k)$ in $C_{\tau-}^{1}\cup\dots\cup C_{\tau-}^{q}$
and $j(k)\in[N]\backslash C_{\tau-}^{1}\cup\dots\cup C_{\tau-}^{q}$
and the jump $C_{\tau}^{r(k)}=C_{\tau-}^{r(k)}\cup\{j(k)\}$is performed.
For example, in Figure \ref{fig:01}, we wait $3T/4$ and then we
add $3$ to the set $C_{(3T/4)-}^{2}=\{2\}$.

The situation in Definition \ref{Def:system} is the following. We
have Poisson processes $N_{i,j}$ like in Subsection \ref{sec:Definition-and-main}.
Let us start at the bottom of the interaction graph and then move
upward.  As the processes $(N_{i,j}(T-t))_{0\leq t\leq T}$ are Poisson
processes, we wait for $\tau'=\inf\{t:\mbox{ jump time of }N_{i,j}(T-.),i\in[q],j\notin[q]\}$.
And then, if $\tau'$ is a jump time for $N_{r',j'}$ with $r'\in[q]$,
we add a branch corresponding to $j'$ to the branch corresponding
to $r'$ (in the same way as in Figure \ref{fig:01}, where we added
the branch with the label $3$ to the branch with the label $2$).
The random times $\tau$ and $\tau'$ have the same law due to Lemma
\ref{Lem:fondamental}, \ref{lem-point-03}. The random couple of
indexes $(r(k),j(k))$, $(r',j')$ have the same law due to Lemma
\ref{Lem:fondamental}, \ref{lem-point-02}.

We now look at the horizontal lines between existing branches (such
as the line between $1$ and $2$ in Figure \ref{fig:01}). Let $0\leq t\leq T$.
We set $j=\#\{k,T_{k}\leq t\}$. We compute
\begin{multline*}
\p(\forall k\leq j,T_{k}\neq T''_{k}|j,T'_{1},T'_{2},\dots,T'_{j},(K_{u})_{0\leq u\leq t})\\
=\p(T'_{1}\frac{\Lambda q(q-1)}{2(N-1)}<V_{1},\dots,(t-T'_{k})\frac{\Lambda(q+j)(q+j-1)}{2(N-1)}<V_{j+1}|j,T'_{1},T'_{2},\dots,T'_{j},(K_{u})_{0\leq u\leq t})\\
=\exp\left(-\int_{0}^{t}\frac{\Lambda K_{u}(K_{u}-1)}{2(N-1)}du\right)\,.
\end{multline*}
 So, conditionally to $(K_{u})_{0\leq u\leq T}$, $(\#\{T_{k}=T''_{k},T_{k}\leq t\})_{t\geq0}$
is an inhomogeneous Poisson process of rate $\left(\frac{\Lambda K_{u}(K_{u}-1)}{2(N-1)}\right)_{0\leq u\leq t}$.
When a jump time of the form $T_{k}=T''_{k}$ occurs, we choose $r(k)$
uniformly in $C_{T_{k}-}^{1}\cup\dots\cup C_{T_{k}-}^{q}$ and $j(k)$
uniformly in $C_{T_{k}-}^{1}\cup\dots\cup C_{T_{k}-}^{q}\backslash\{r(k)\}$.
We then add a horizontal line between branches $r(k)$ and $j(k)$.
Due to Lemma \ref{Lem:fondamental}, \ref{lem-point-02}, this is
the way horizontal branches are added to existing vertical branches
in Definition \ref{Def:system}. $ $

\end{proof}

\subsection{Proof of Lemma\label{sec:Proof-of-Lemma} \ref{Lem:law-K-tilde}}

\begin{proof}The process $(\widetilde{K}_{t})_{0\leq t\leq T}$ is
piecewise constant and has jumps of size $1$. The jump times of $\widetilde{K}_{t}$
belong to $\left\{ T'_{k},k\geq1\right\} $ or to $\left\{ \widetilde{T}'_{k},k\geq1\right\} $.
The jump times of $K_{t}$ belong to $\left\{ T'_{k},k\geq1\right\} $.
Suppose we are at time $s$, and we know $\widetilde{K}_{s}$. %
{} We set $\widetilde{j}=\#\{\widetilde{T}_{k},s<\widetilde{T}_{k}\leq t\}$,
$\widetilde{k}_{0}=\sup\{k,\widetilde{T}_{k}\leq s\}$, $j=\#\{T_{k},s<T_{k}\leq t\}$,
$k_{0}=\sup\{k,T_{k}\leq s\}$. We compute
\begin{multline*}
\E(\1_{\widetilde{K}_{t}=\widetilde{K}_{s}}|\widetilde{K}_{s})=\E(\E(\1_{\widetilde{K}_{t}=\widetilde{K}_{s}}|\widetilde{K}_{s},K_{s},T_{k_{0}},\dots,T_{k_{0}+j},\widetilde{T}_{\widetilde{k}_{0}},\dots,\widetilde{T}_{\widetilde{k}_{0}+\widetilde{j}})|\widetilde{K}_{s})=\\
\E\left[\E\left[\prod_{i=1}^{j}\1_{U_{k_{0}+i}>(T_{k_{0}+i}-T_{k_{0}+i-1})\frac{\Lambda K_{s}(N-K_{s})_{+}}{N-1}}\right.\right.\\
\times\1_{U_{k_{0}+j}>(t-T_{k_{0}+j})\frac{\Lambda K_{s}(N-K_{s})_{+}}{N-1}}\\
\times\prod_{i=1}^{\widetilde{j}}\1_{\widetilde{U}_{\widetilde{k}_{0}+i}>(\widetilde{T}_{\widetilde{k}_{0}+i}-\widetilde{T}_{\widetilde{k}_{0}+i-1})\left(\Lambda\widetilde{K}_{s}-\frac{\Lambda K_{s}(N-K_{s})_{+}}{N-1}\right)}\\
\times\1_{\widetilde{U}_{\widetilde{k}_{0}+\widetilde{j}}>(t-\widetilde{T}_{\widetilde{k}_{0}+\widetilde{j}})\left(\Lambda\widetilde{K}_{s}-\frac{\Lambda K_{s}(N-K_{s})_{+}}{N-1}\right)}\\
\left.\left.|\widetilde{K}_{s},K_{s},T_{k_{0}},\dots,T_{k_{0}+j},\widetilde{T}_{\widetilde{k}_{0}},\dots,\widetilde{T}_{\widetilde{k}_{0}+\widetilde{j}}\right]|\widetilde{K}_{s}\right]=\\
\exp(-(t-s)\Lambda\widetilde{K}_{s})\,.
\end{multline*}

The process $(\widetilde{L}_{t})_{0\leq t\leq T}$ is piecewise constant
and has jumps of size $1$. The jump times of this process belong
to $\{T''_{k},k\geq1\}$, or to $\{\widetilde{T}''_{k},k\geq1\}$.
Suppose we are at time $s$ and we know $\widetilde{L}_{s}$, $(\widetilde{K}_{u})_{0\leq u\leq T}$.
Let $t\geq s$. %
Due to the properties of the exponential law, the probability $\p(\widetilde{L}_{t}=\widetilde{L}_{s}|\widetilde{L}_{s},(\widetilde{K}_{u})_{0\leq u\leq T})$
is equal to %
{} This probability is equal to
\begin{multline*}
\E(\1_{\widetilde{L}_{t}=\widetilde{L}_{s}}|\widetilde{L}_{s},(\widetilde{K}_{u})_{0\leq u\leq T})\\
=\E(\E(\1_{\widetilde{L}_{t}=\widetilde{L}_{s}}|\widetilde{L}_{s},(\widetilde{K}_{u})_{0\leq u\leq T},T_{k_{0}},\dots,T_{k_{0}+j},\widetilde{T}_{\widetilde{k}_{0}},\dots,\widetilde{T}_{\widetilde{k}_{0}+\widetilde{j}})|\widetilde{L}_{s},(\widetilde{K}_{u})_{0\leq u\leq T})\\
=\E\left[\E\left[\prod_{i=1}^{j}\1_{V_{k_{0}+i}>(T_{k_{0}+i}-T_{k_{0}+i-1})\frac{\Lambda K_{_{T_{k_{0}+i}}}(K_{T_{k_{0}+i}}-1)}{N-1}}\right.\right.\\
\times\1_{V_{k_{0}+j}>(t-T_{k_{0}+j})\frac{\Lambda K_{_{T_{k_{0}+j}}}(K_{T_{k_{0}+j}}-1)}{N-1}}\\
\times\prod_{i=1}^{\widetilde{j}}\1_{\widetilde{U}'_{\widetilde{k}_{0}+i}>(\widetilde{T}_{\widetilde{k}_{0}+i}-\widetilde{T}_{\widetilde{k}_{0}+i-1})\left(\frac{\Lambda\widetilde{K}_{_{\widetilde{T}_{k_{0}+i}}}(\widetilde{K}_{_{\widetilde{T}_{k_{0}+i}}}-1)}{N-1}-\frac{\Lambda K_{\widetilde{T}_{k_{0}+i}}(K_{\widetilde{T}_{k_{0}+i}}-1)_{+}}{N-1}\right)}\\
\times\1_{\widetilde{U}'_{\widetilde{k}_{0}+\widetilde{j}}>(t-\widetilde{T}_{\widetilde{k}_{0}+\widetilde{j}})\left(\frac{\Lambda\widetilde{K}_{_{\widetilde{T}_{k_{0}+i}}}(\widetilde{K}_{_{\widetilde{T}_{k_{0}+i}}}-1)}{N-1}-\frac{\Lambda K_{\widetilde{T}_{k_{0}+i}}(K_{\widetilde{T}_{k_{0}+i}}-1)_{+}}{N-1}\right)}\\
\left.\left.|\widetilde{L}_{s},(\widetilde{K}_{u})_{0\leq u\leq T},T_{k_{0}},\dots,T_{k_{0}+j},\widetilde{T}_{\widetilde{k}_{0}},\dots,\widetilde{T}_{\widetilde{k}_{0}+\widetilde{j}}\right]|\widetilde{L}_{s},(\widetilde{K}_{u})_{0\leq u\leq T}\right]=\\
\exp(-(t-s)\Lambda\widetilde{K}_{s})\,.
\end{multline*}
 This probability is equal to $\p\left(\int_{s}^{t}\frac{\Lambda\widetilde{K}_{u}(\widetilde{K}_{u}-1)}{N-1}du\leq V'_{1}|(\widetilde{K}_{u})_{0\leq u\leq T}\right)$
(for some $V'_{1}$ of law $\mathcal{E}(1)$). This proves the point
\ref{loiKtilde:02} of the lemma.

We have for all $\omega,t$, $\Delta K_{t}(\omega)=1\Rightarrow\Delta\widetilde{K}_{t}(\omega)=1$
and $\Delta L_{t}(\omega)=1\Rightarrow\Delta\widetilde{L}_{t}(\omega)=1$,
so we have the point \ref{lemKtilde:03} of the Lemma. The point \ref{lemKtilde:04}
of the Lemma is immediate.

Let $k\geq1$. Suppose we are at time $\widetilde{T}_{k-1}$, with
$\widetilde{T}_{k-1}<T$. The variables $\inf\{T'_{l},T'_{l}\geq\widetilde{T}_{k-1}\}-\widetilde{T}_{k-1}$,
$\inf\{T''_{l},T''_{l}\geq\widetilde{T}_{k-1}\}-\widetilde{T}_{k-1}$,
$\inf\{\widetilde{T}'_{l},\widetilde{T}'_{l}\geq\widetilde{T}_{k-1}\}-\widetilde{T}_{k-1}$,
$\inf\{\widetilde{T}''_{l},\widetilde{T}''_{l}\geq\widetilde{T}_{k-1}\}-\widetilde{T}_{k-1}$
are of exponential law (recall the definition from sections \ref{sub:Backward-point-of},
\ref{sub:Auxiliary-systems}). The infimum of four independent exponential
variables $E_{1},\dots,E_{4}$ of parameters, respectively, $\lambda_{1},\dots,\lambda_{4}$
satisfies $\p(E_{1}=\inf(E_{1},\dots,E_{4})|\inf(E_{1},\dots,E_{4})<t)=\frac{\lambda_{1}}{\lambda_{1}+\dots+\lambda_{4}}$
($\forall t>0$) (see Th. 2.3.3. of \cite{norris-1998}). So,
\begin{eqnarray*}
\p(\widetilde{T}_{k}\in\{\widetilde{T}'_{l},l\geq1\}|\mathcal{K}_{\widetilde{T}_{k-1}},\widetilde{\mathcal{K}}_{\widetilde{T}_{k-1}},\widetilde{T}_{k}<T) & = & \frac{\Lambda\widetilde{K}_{\widetilde{T}_{k-1}}-\frac{\Lambda K_{\widetilde{T}_{k-1}}(N-K_{\widetilde{T}_{k-1}})_{+}}{N-1}}{\Lambda\widetilde{K}_{\widetilde{T}_{k-1}}+\frac{\Lambda\widetilde{K}_{\widetilde{T}_{k-1}}(\widetilde{K}_{\widetilde{T}_{k-1}}-1)}{(N-1)}}\,,
\end{eqnarray*}
 
\[
\p(\widetilde{T}_{k}\in\{T_{l}',l\geq1\}|\mathcal{K}_{\widetilde{T}_{k-1}},\widetilde{\mathcal{K}}_{\widetilde{T}_{k-1}},\widetilde{T}_{k}<T)=\frac{\frac{\Lambda K_{\widetilde{T}_{k-1}}(N-K_{\widetilde{T}_{k-1}})_{+}}{N-1}}{\Lambda\widetilde{K}_{\widetilde{T}_{k-1}}+\frac{\Lambda\widetilde{K}_{\widetilde{T}_{k-1}}(\widetilde{K}_{\widetilde{T}_{k-1}}-1)}{(N-1)}}\,,
\]
 so, recalling (\ref{eq:tirage-r-j-01}), (\ref{eq:proba-choix-01}),
(\ref{eq:proba-choix-02}), 
\begin{multline}
\p(\Delta\widetilde{K}_{\widetilde{T}_{k}}^{i}=1|\widetilde{T}_{k}\in\{T'_{l},\widetilde{T}'_{l},l\geq1\},\mathcal{K}_{\widetilde{T}_{k-1}},\widetilde{\mathcal{K}}_{\widetilde{T}_{k-1}},\widetilde{T}_{k}<T)=\\
\left(\frac{\widetilde{K}_{\widetilde{T}_{k-1}}-\frac{K_{\widetilde{T}_{k-1}}(N-K_{\widetilde{T}_{k-1}})_{+}}{N-1}}{\widetilde{K}_{\widetilde{T}_{k-1}}}\right)\times\left[\frac{K_{\widetilde{T}_{k-1}}-\frac{K_{\widetilde{T}_{k-1}}(N-K_{\widetilde{T}_{k-1}})_{+}}{N-1}}{\widetilde{K}_{\widetilde{T}_{k-1}}-\frac{K_{\widetilde{T}_{k-1}}(N-K_{\widetilde{T}_{k-1}})_{+}}{N-1}}\times\frac{K_{\widetilde{T}_{k-1}}^{i}}{K_{\widetilde{T}_{k-1}}}\right.\\
\left.+\frac{\widetilde{K}_{\widetilde{T}_{k-1}}-K_{\widetilde{T}_{k-1}}}{\widetilde{K}_{\widetilde{T}_{k-1}}-\frac{K_{\widetilde{T}_{k-1}}(N-K_{\widetilde{T}_{k-1}})_{+}}{N-1}}\times\frac{\widetilde{K}_{\widetilde{T}_{k-1}}^{i}-K_{\widetilde{T}_{k-1}}^{i}}{\widetilde{K}_{\widetilde{T}_{k-1}}-K_{\widetilde{T}_{k-1}}}\right]+\frac{\left(\frac{K_{\widetilde{T}_{k-1}}(N-K_{\widetilde{T}_{k-1}})_{+}}{N-1}\right)}{\widetilde{K}_{\widetilde{T}_{k-1}}}\times\frac{K_{\widetilde{T}_{k-1}}^{i}}{K_{\widetilde{T}_{k-1}}}=\frac{\widetilde{K}_{\widetilde{T}_{k-1}}^{i}}{\widetilde{K}_{\widetilde{T}_{k-1}}}\,,\label{eq:long-calcul-01}
\end{multline}
 and this last expression depends only on $\widetilde{\mathcal{K}}_{\widetilde{T}_{k}-}$.
By point \ref{loiKtilde:02} of the lemma, the process $(\widetilde{K}_{s})_{s\geq0}$
is equal in law to the sum of $q$ independent Yule processes $Y_{s}^{(1)}$,
\dots, $Y_{s}^{(q)}$, and its law is thus independent of $N$ (see
\cite{athreya-ney-2004}, p. 102-109, p. 109 for the law of the Yule
process). We have, for all $s$, 
\begin{equation}
\p(Y_{s}^{(1)}=k)=e^{-s\Lambda}(1-e^{-s\Lambda})^{k-1}\label{eq:law-Yule}
\end{equation}
 and so (see for example \cite{dodge-2004}, p. 288), 
\begin{equation}
\p(\widetilde{K}_{t}=k)=\p(Y_{t}^{(1)}+\dots+Y_{t}^{(q)}=k)=\left(\begin{array}{c}
k\\
q-1
\end{array}\right)(e^{-\Lambda t})^{q}(1-e^{-\Lambda t})^{k-q}\,.\label{Eq:majK}
\end{equation}
  Using the point \ref{loiKtilde:02} of the lemma, Equation (\ref{eq:long-calcul-01})
and the point \ref{lem-point-02} of Lemma \ref{Lem:fondamental},
we obtain the point \ref{lemKtilde:01} of the Lemma.

Reasoning as above, we can show that, conditionally to $\mathcal{K}_{T}$,
$\widetilde{\mathcal{K}}_{T}$, the process $(\#\{T_{k},T_{k}=T''_{k},T_{k}\leq t\})_{0\leq t\leq T}$
is a homogeneous Poisson process of rate $\left(\frac{\Lambda K_{t}(K_{t}-1)}{2(N-1)}\right)_{0\leq t\leq T}$
and the process $(\#\{\widetilde{T}_{k},\widetilde{T}_{k}=\widetilde{T}''_{k},\widetilde{T}_{k}\leq t\})_{0\leq t\leq T}$
is a homogeneous Poisson process of rate 
\[
\left(\frac{\Lambda\widetilde{K}_{t}(\widetilde{K}_{t}-1)-\Lambda K_{t}(K_{t}-1)}{2(N-1)}\right)_{0\leq t\leq T}\,.
\]
So, by Lemma \ref{Lem:fondamental}, \ref{lem-point-03}, 
\[
\p(t\in\{T''_{l},l\geq1\}|\mathcal{K}_{T},\widetilde{\mathcal{K}}_{T},\Delta\widetilde{L}_{t}=1)=\frac{\Lambda K_{t}(K_{t}-1)}{\Lambda\widetilde{K}_{t}(\widetilde{K}_{t}-1)}\,.
\]
 We have, for all $i$ (recalling (\ref{eq:tirage-r-j-02})), 
\[
\p(\Delta\widetilde{L}_{t}^{\{i,i+1\}}=1|\mathcal{K}_{T},\widetilde{\mathcal{K}}_{T},t\in\{T''_{l},l\geq1\})=\frac{2K_{t}^{i}K_{t}^{i+1}}{K_{t}(K_{t}-1)}\,,
\]
 and (recalling (\ref{eq:proba-saut-en-plus}), (\ref{eq:proba-saut-en-plus-02}))
\begin{multline*}
\p(\Delta\widetilde{L}_{t}^{\{i,i+1\}}=1|\mathcal{K}_{T},\widetilde{\mathcal{K}}_{T},t\in\{\widetilde{T}''_{l},l\geq1\})=\\
\frac{(\widetilde{K}_{t}-K_{t})K_{t}+(\widetilde{K}_{t}-K_{t})(\widetilde{K}_{t}-K_{t}-1)}{\ktt(\ktt-1)-\kt(\kt-1)}\times\left(\frac{(\ktt^{i}-\kt^{i})\ktt^{i+1}}{(\ktt-\kt)(\ktt-1)}+\frac{(\ktt^{i+1}-\kt^{i+1})\ktt^{i}}{(\ktt-\kt)(\ktt-1)}\right)\\
+\frac{(\ktt-\kt)K_{t}}{\ktt(\ktt-1)-\kt(\kt-1)}\times\left(\frac{(\ktt^{i}-\kt^{i})K_{t}^{i+1}}{(\ktt-\kt)\kt}+\frac{(\ktt^{i+1}-\kt^{i+1})\kt^{i}}{(\ktt-\kt)\kt}\right)=\\
\frac{(\ktt^{i}-\kt^{i})(\ktt^{i+1}+\kt^{i+1})+(\ktt^{i+1}-\kt^{i+1})(\ktt^{i}+\kt^{i})}{\ktt(\ktt-1)-\kt(\kt-1)}\,.
\end{multline*}
 So,
\begin{multline*}
\p(\Delta\widetilde{L}_{t}^{\{i,i+1\}}=1|\mathcal{K}_{T},\widetilde{\mathcal{K}}_{T},\Delta\widetilde{L}_{t}=1)=\\
\left(\frac{\ktt(\ktt-1)-\kt(\kt-1)}{\ktt(\ktt-1)}\right)\times\left(\frac{(\ktt^{i}-\kt^{i})(\ktt^{i+1}+\kt^{i+1})+(\ktt^{i+1}-\kt^{i+1})(\ktt^{i}+\kt^{i})}{\ktt(\ktt-1)-\kt(\kt-1)}\right)\\
+\left(\frac{\kt(\kt-1)}{\ktt(\ktt-1)}\right)\left(\frac{2\kt^{i}\kt^{i+1}}{\kt(\kt-1)}\right)=\frac{2\ktt^{i}\ktt^{i+1}}{\ktt(\ktt-1)}\,,
\end{multline*}
 so, 
\[
\p(\Delta\widetilde{L}_{t}^{\{i,i+1\}}=1|\widetilde{\mathcal{K}}_{T},\Delta\widetilde{L}_{t}=1)=\frac{2\ktt^{i}\ktt^{i+1}}{\ktt(\ktt-1)}\,.
\]
 So, using Lemma \ref{Lem:fondamental}, \ref{lem-point-02} we have
the point \ref{lemKtilde:05} of the lemma.

\end{proof}

The author would  like to thank the following colleagues for their
input: Christophe Giraud, Nicolas Champagnat, Benjamin Jourdain, Tony
Lelièvre, Patricia Reynaud-Bouret. 

\bibliographystyle{amsalpha.bst} \bibliographystyle{amsalpha}
\bibliography{bib-dl-boltzmann}

\newcommand{\SortNoop}[1]{}
\providecommand{\bysame}{\leavevmode\hbox to3em{\hrulefill}\thinspace}
\providecommand{\MR}{\relax\ifhmode\unskip\space\fi MR }
\providecommand{\MRhref}[2]{%
  \href{http://www.ams.org/mathscinet-getitem?mr=#1}{#2}
}
\providecommand{\href}[2]{#2}
\begin{thebibliography}{{\SortNoop{Pena}}dlPG99}

\bibitem[AN72]{athreya-ney-2004}
Krishna~B. Athreya and Peter~E. Ney, \emph{Branching processes},
  Springer-Verlag, New York, 1972, Die Grundlehren der mathematischen
  Wissenschaften, Band 196. \MR{MR0373040 (51 \#9242)}

\bibitem[{\SortNoop{Daw}}DZ91]{dawson-zheng-1991}
D.~A. {\SortNoop{Daw}}~Dawson and X.~Zheng, \emph{Law of large numbers and
  central limit theorem for unbounded jump mean-field models}, Adv. in Appl.
  Math. \textbf{12} (1991), no.~3, 293--326. \MR{1117994 (92k:60220)}

\bibitem[DPR09]{del-moral-patras-rubenthaler-2008}
Pierre {Del Moral}, Fr{\'e}d{\'e}ric Patras, and Sylvain Rubenthaler,
  \emph{Tree based functional expansions for {F}eynman-{K}ac particle models},
  Ann. Appl. Probab. \textbf{19} (2009), no.~2, 778--825. \MR{MR2521888
  (2010c:47110)}

\bibitem[DPR11]{del-moral-patras-rubenthaler-2009}
\bysame, \emph{Convergence of $u$-statistics for interacting particle systems},
  Journal of Theoretical Probability \textbf{24} (2011), no.~4, 1002--1027.

\bibitem[{\SortNoop{Dzd}}Do04]{dodge-2004}
Yadolah {\SortNoop{Dzd}}Dodge, \emph{Statistique}, second ed., Springer-Verlag,
  Paris, 2004, Dictionnaire encyclop{\'e}dique. [Encyclopedic dictionary].
  \MR{2117847 (2005h:62001)}

\bibitem[GM94]{graham-meleard-1994}
Carl Graham and Sylvie M{\'e}l{\'e}ard, \emph{Chaos hypothesis for a system
  interacting through shared resources}, Probab. Theory Related Fields
  \textbf{100} (1994), no.~2, 157--173. \MR{MR1296426 (95j:60165)}

\bibitem[GM97]{graham-meleard-1997}
\bysame, \emph{Stochastic particle approximations for generalized {B}oltzmann
  models and convergence estimates}, Ann. Probab. \textbf{25} (1997), no.~1,
  115--132.

\bibitem[GM99]{graham-meleard-1999}
\bysame, \emph{Probabilistic tools and {M}onte-{C}arlo approximations for some
  {B}oltzmann equations}, CEMRACS 1999 (Orsay), ESAIM Proc., vol.~10, Soc.
  Math. Appl. Indust., Paris, 1999, pp.~77--126 (electronic). \MR{MR1865189
  (2003a:82062)}

\bibitem[Kin93]{kingman-1993}
John F.~C. Kingman, \emph{Poisson processes}, Oxford Studies in Probability,
  vol.~3, The Clarendon Press Oxford University Press, New York, 1993, Oxford
  Science Publications. \MR{1207584 (94a:60052)}

\bibitem[Lee90]{lee-1990}
Alan~J. Lee, \emph{{$U$}-statistics}, Statistics: Textbooks and Monographs,
  vol. 110, Marcel Dekker Inc., New York, 1990, Theory and practice.
  \MR{MR1075417 (91k:60026)}

\bibitem[M{\'e}l98]{meleard-1998}
Sylvie M{\'e}l{\'e}ard, \emph{Convergence of the fluctuations for interacting
  diffusions with jumps associated with {B}oltzmann equations}, Stochastics
  Stochastics Rep. \textbf{63} (1998), no.~3-4, 195--225. \MR{MR1658082
  (99g:60103)}

\bibitem[Nor98]{norris-1998}
J.~R. Norris, \emph{Markov chains}, Cambridge Series in Statistical and
  Probabilistic Mathematics, vol.~2, Cambridge University Press, Cambridge,
  1998, Reprint of 1997 original. \MR{1600720 (99c:60144)}

\bibitem[NS06]{nica-speicher-2006}
Alexandru Nica and Roland Speicher, \emph{Lectures on the combinatorics of free
  probability}, London Mathematical Society Lecture Note Series, vol. 335,
  Cambridge University Press, Cambridge, 2006. \MR{2266879 (2008k:46198)}

\bibitem[{\SortNoop{Pena}}dlPG99]{delapena-gine-1999}
V{\'{\i}}ctor~H. {\SortNoop{Pena}}~de~la Pe{\~n}a and Evarist Gin{\'e},
  \emph{Decoupling}, Probability and its Applications (New York),
  Springer-Verlag, New York, 1999, From dependence to independence, Randomly
  stopped processes. $U$-statistics and processes. Martingales and beyond.
  \MR{MR1666908 (99k:60044)}

\bibitem[ST85]{shiga-tanaka-1985}
Tokuzo Shiga and Hiroshi Tanaka, \emph{Central limit theorem for a system of
  {M}arkovian particles with mean field interactions}, Z. Wahrsch. Verw.
  Gebiete \textbf{69} (1985), no.~3, 439--459. \MR{787607 (88a:60056)}

\bibitem[Szn84]{sznitman-1984}
Alain-Sol Sznitman, \emph{Nonlinear reflecting diffusion process, and the
  propagation of chaos and fluctuations associated}, J. Funct. Anal.
  \textbf{56} (1984), no.~3, 311--336. \MR{743844 (86b:60167)}

\bibitem[Szn85]{sznitman-1985}
A.-S. Sznitman, \emph{A fluctuation result for nonlinear diffusions},
  Infinite-dimensional analysis and stochastic processes ({B}ielefeld, 1983),
  Res. Notes in Math., vol. 124, Pitman, Boston, MA, 1985, pp.~145--160.
  \MR{865024 (88i:60116)}

\bibitem[Uch83a]{uchiyama-1983-a}
K{\=o}hei Uchiyama, \emph{A fluctuation problem associated with the {B}oltzmann
  equation for a gas of molecules with a cutoff potential}, Japan. J. Math.
  (N.S.) \textbf{9} (1983), no.~1, 27--53. \MR{722535 (85c:82040)}

\bibitem[Uch83b]{uchiyama-1983-b}
\bysame, \emph{Fluctuations of {M}arkovian systems in {K}ac's caricature of a
  {M}axwellian gas}, J. Math. Soc. Japan \textbf{35} (1983), no.~3, 477--499.
  \MR{702771 (84m:60119)}

\bibitem[Uch88]{uchiyama-1988}
\bysame, \emph{Fluctuations in a {M}arkovian system of pairwise interacting
  particles}, Probab. Theory Related Fields \textbf{79} (1988), no.~2,
  289--302. \MR{958292 (89i:60198)}

\end{thebibliography}

\end{document}